\documentclass[a4paper,11pt]{amsart}
 \usepackage{amsrefs}
\usepackage{amsmath}
\usepackage{amsthm}
\usepackage{mathtools}
\usepackage{amssymb}
\usepackage{amsfonts}
\usepackage{xcolor}
\usepackage[bookmarks=true,hyperindex,pdftex,colorlinks,citecolor=blue]{hyperref}
\usepackage{enumitem}
\usepackage{marginnote}
\usepackage{soul}
\usepackage[normalem]{ulem}

\usepackage[a4paper,lmargin=3cm,rmargin=3cm,tmargin=4cm,bmargin=4cm,marginparwidth=2.8cm,marginparsep=1mm]{geometry}

\theoremstyle{plain}
\newtheorem{theorem}{Theorem}[section]
\newtheorem*{thmA}{Theorem A}
\newtheorem*{thmB}{Theorem B}
\newtheorem*{theorem*}{Theorem}

\newtheorem{lemma}[theorem]{Lemma}
\newtheorem{corollary}[theorem]{Corollary}
\newtheorem{proposition}[theorem]{Proposition}

\theoremstyle{definition}
\newtheorem*{definition*}{Definition}
\newtheorem{definition}[theorem]{Definition}
\newtheorem{example}[theorem]{Example}

\newtheorem{assumption}[theorem]{Assumption}
\newtheorem{convention}[theorem]{Convention}
\newtheorem{claim}[theorem]{Claim}
\newtheorem*{acknowledgement*}{Acknowledgement}
\newtheorem*{Case1}{\textbf{Case 1}}
\newtheorem*{Case2}{\textbf{Case 2}}

\newtheorem*{Step1}{\textbf{Step 1}}
\newtheorem*{Step2}{\textbf{Step 2}}
\newtheorem*{Step3}{\textbf{Step 3}}
\newtheorem*{Step4}{\textbf{Step 4}}

\theoremstyle{remark}
\newtheorem{remark}[theorem]{Remark}

\numberwithin{equation}{section}

\def \R {\mathbb{R}}

\def \M {\mathcal{M}}
\def \N {\mathbb{N}}

\def \U {\mathcal{U}}

\def \F {\mathcal{F}}

\def \eps {\varepsilon}

\def \sbs {\subseteq}
\def \sps {\supseteq}

\DeclareMathOperator{\dist}{dist}
\DeclareMathOperator{\diam}{diam}
\DeclareMathOperator{\Lip}{Lip}

\DeclareMathOperator{\supp}{supp}

\DeclareMathOperator{\spn}{span}
\DeclareMathOperator{\Id}{Id}

\DeclareMathOperator{\gex}{gex}
\DeclareMathOperator{\rad}{rad}

\begin{document}

\title[$L$-embeddability of Lipschitz-free spaces]{On the $L$-embeddability of Lipschitz-free spaces  in their biduals}%

\date{\today} 

\author[Braga]{Bruno M. Braga}
\address[B. M. Braga]{IMPA, Estrada Dona Castorina 110, 22460-320, Rio de Janeiro, Brazil}
\email{demendoncabraga@gmail.com}
\urladdr{\url{https://sites.google.com/site/demendoncabraga}}
\thanks {B. M. Braga  was partially supported by FAPERJ, grant E-26/204.317/2025,  by CNPq, grant 303182/2026-1, and by Serrapilheira, grant R-2501-51476.}

\author[Gartland]{C. Gartland}
\address[C. Gartland]{Department of Mathematics and Statistics, University of
North Carolina at Charlotte, Charlotte, North Carolina, USA}
\email{cgartla1@charlotte.edu}
\urladdr{\url{https://chrisgartland.wordpress.com/}}
\thanks{C. Gartland was partially supported by NSF Award DMS-2555144}

\author[Lancien]{G. Lancien}
\address[G. Lancien]{Universit\'e Marie et Louis Pasteur, CNRS, Laboratoire de mathématiques de Besançon  (UMR 6623), F-25000 Besan\c con, France.}
\email{gilles.lancien@umlp.fr}
\urladdr{\url{https://glancien.pages.math.cnrs.fr/}}
\thanks{G. Lancien was partially supported by the French ANR project No. ANR-24-CE40-0892-01.}

\author[Motakis]{P. Motakis}
\address[P. Motakis]{Department of Mathematics and Statistics, York University, 4700 Keele Street, Toronto, Ontario, M3J 1P3, Canada}
\email{pmotakis@yorku.ca}
\urladdr{\url{https://pmotakis.mathstats.yorku.ca/}}
\thanks{P. Motakis was partially supported by NSERC Grant RGPIN-2021-03639.}

\author[Perneck\'a]{E. Perneck\'a}
\address[E. Perneck\'a]{Faculty of Information Technology, Czech Technical University in Prague, Th\'akurova 9, 160 00, Prague 6, Czech Republic}
\email{eva.pernecka@fit.cvut.cz}
\urladdr{\url{https://fit.cvut.cz/en/faculty/people/5300-eva-pernecka}}
\thanks{}

\author[Schlumprecht]{Th. Schlumprecht}
\address[Th. Schlumprecht]{Texas A\&M University, College Station, TX 77843,
USA and Faculty of Electrical Engineering, Czech Technical University in Prague, Zikova 4, 166 27, Prague}
\email{t-schlumprecht@tamu.edu}
\urladdr{\url{https://people.tamu.edu/~t-schlumprecht/}}
\thanks{Th. Schlumprecht was partially supported by NSF Award DMS-2349322}

\begin{abstract}
A Banach space $X$ is \emph{$L$-embedded in its bidual} if there is a projection $P\colon X^{**}\to X$ such that $\|x\|=\|Px\|+\|x-Px\|$ for all $x\in X^{**}$. We show that, as long as $X$ has dimension at least $2$, its Lipschitz-free space, denoted by $\F(X)$, is not $L$-embedded in its bidual. Our methods have applications to the problem of when the $L$-embeddability of $\F(M)$ passes to $\F(A)$ for a metric space $M$ and  $A\sbs M$. This is the case when $A$ is compact or when $A$ is geodesically closed.
\end{abstract}

\maketitle

\noindent\textbf{AI usage statement.}
AI was not used to develop the ideas nor to generate or verify proofs in this project. Its use was restricted to improving selected text excerpts and locating references. In particular, ChatGPT helped locate a reference establishing that reversible $n$-dimensional Finsler manifolds admit bi-Lipschitz embeddings of unit balls of $n$-dimensional normed spaces. This reference was used in obtaining Corollary~\ref{finsler.lipschitz.free.space.non.L-embeddability}.

\tableofcontents
\section{Introduction}

One of the most important open questions in the nonlinear theory of Banach spaces asks about the Lipschitz rigidity of $\ell_1$: if a Banach space is Lipschitz isomorphic to $\ell_1$, must it be linearly isomorphic to $\ell_1$? This   problem is related to complementability of Lipschitz-free spaces in their biduals, via the well known implication asserting that if the Lipschitz-free space $\F(\ell_1)$ of $\ell_1$ is complemented in $\F(\ell_1)^{**}$, then the Lipschitz rigidity question for $\ell_1$ has a positive answer (see \cite[Section 4.7]{BaudierLancienBook}).

Currently, the state of the art on the problem of complementability of free spaces is due to C\'uth, Kalenda, and Kaplick\'{y}, who showed that $\F(X)$ is always complemented in $\F(X)^{**}$ whenever $X$ is a finite-dimensional Banach space (see \cite[Theorem 1.1]{CuthKalendaKaplicky2019Israel}). Moreover, they showed that when $X$ is the finite-dimensional Euclidean space $\ell_2^n$, the free space $\F(\ell_2^n)$ is even $1$-complemented in its bidual, and hence $\F(\ell_1^n)$ is $\sqrt{n}$-complemented in $\F(\ell_1^n)^{**}$. However, even with the uniform (in $n$) complementability of $\F(\ell_2^n)$ in $\F(\ell_2^n)^{**}$, the corresponding infinite-dimensional question of whether $\F(\ell_2)$ is complemented in $\F(\ell_2)^{**}$ remains unknown at present. Since the known bounds on the complementation constants of $\F(\ell_1^n)$ in $\F(\ell_1^n)^{**}$ tend to $\infty$ as $n\to\infty$, the problem of complementability of $\F(\ell_1)$ in $\F(\ell_1)^{**}$ is seemingly even more difficult.

In this paper, we study a stronger form of complementability of a Banach space into its bidual called \emph{$L$-embeddability}. Our main result shows that this property is highly rigid: for a Banach space $X$, the free space $\F(X)$ is $L$-embedded in its bidual if and only if $\dim(X)\leq 1$. In particular, $\F(\ell_1)$ is not $L$-embedded in its bidual, and so this stronger question about complementability of $\F(\ell_1)$ in $\F(\ell_1)^{**}$ has a negative answer. A couple of comments are in order. Firstly,
this result was already known for certain Banach spaces; Proch\'azka and Rueda Zoca showed in \cite[Corollary 4.6]{ProchazkaZoca2018} that $\F(\ell_p)$ is not $L$-embedded in its bidual for any $1<p<\infty$. They achieved it by embedding into $\ell_p$, for $1<p<\infty$, an infinite bounded uniformly discrete metric space whose Lipschitz-free space does not have octahedral norm. This rules out the $L$-embeddability of $\F(\ell_p)$ in its bidual as any non-reflexive subspace of a separable space $L$-embedded in its bidual space must have octahedral norm by \cite[Proposition 3.4]{Zidi2013}. 
On the other hand, the authors proved in \cite[Proposition 4.8]{ProchazkaZoca2018} that $\ell_1$ does not contain any such metric space, and it clearly holds that no finite-dimensional Banach space can contain such a space either. Their methods therefore cannot be applied in these cases. Our work relies on studying the supports of functionals on Lipschitz spaces and their behavior under $L$-projections, which leads to a complete answer. Secondly, it is well-known that $L$-embeddability implies weak sequential completeness (see \cite[Theorem 4]{Godefroy1983AnnSciEcSup} or \cite[Lemma 4]{Godefroy1984TAMS} for an elementary proof) and, even more, it implies Pe\l czy\'nski's property (V*) (see \cite[Theorems IV.2.2  and IV.2.7]{HarmandWernerWernerBookMideals}). 
  It would therefore be natural to rule out $L$-embeddability of $\F(X)$ by demonstrating the failure of either of these properties. However, this is not possible in many cases, as we know that $\F(X)$ is weakly sequentially complete if $X$ is superreflexive (see \cite[Corollary 2.5]{AliagaNousPetitjeanProchazka2021Studia}) and that $\F(X)$ has Pe\l cz\'nski's property (V*) for all Hilbert spaces $X$ (see \cite[Theorem C]{AliagaPerneckaQuero2025ResultsMath}).  
In particular, our main result additionally implies that $L$-embeddability and Pe\l cz\'nski's property (V*) are distinct for Lipschitz-free spaces of finite-dimensional Banach spaces. Finally, let us emphasize that it is not known whether $\F(\ell_1)$ is weakly sequentially complete.

We now briefly recall $L$-embeddability and Lipschitz-free spaces. Once this is covered, we finish this introduction by describing some of our main results.

\subsection{$L$-embeddability} 
 A \emph{projection} on a Banach space $X$ is a bounded linear operator  $P\colon X\to X$  with $P^2=P$. If moreover $P$ satisfies 
    \begin{equation}\label{Defi.Lprof.intro}\|x\|=\|Px\|+\|x-Px\|,\ \text{ for all }\ x\in X,
    \end{equation}
    then $P$ is called an \emph{$L$-projection} and its image is called an \emph{$L$-summand} (of $X$). These objects were introduced in \cite{Cunningham1960TAMS} and were further investigated by Alfsen and Effros in the fundamental work \cite{AlfsenEffros1972Annals}.

A Banach space $X$ is said to be \emph{$L$-embedded in its bidual} if there exists an $L$-projection from $X^{**}$ onto $X$. Clearly, $\ell_1$ is $L$-embedded in $\ell_1^{**}$ and, more generally, so is any other $L_1$-space. Other examples include the Hardy space $H^1_0$ and the dual of the disk algebra (see \cite[Section IV.1]{HarmandWernerWernerBookMideals}). We refer to the monographs \cite{Behrends1979LectureNotesSpringer,HarmandWernerWernerBookMideals} for detailed treatments of $L$-projections.

\subsection{Lipschitz-free spaces} \label{ss:LF}
The Lipschitz-free space is a Banach space constructed over a metric space, serving as an apparatus for linearizing Lipschitz maps. Precisely, given a \emph{pointed metric space} $(M,d,0)$, i.e., a metric space $(M,d)$ with a base point $0\in M$, we let 
\[\Lip_0(M)
\coloneqq\left\{f\colon M\to \R\colon f\ \text{ is Lipschitz and }\ f(0)=0\right\}\]
and endow $\Lip_0(M)$ with the norm given by the \emph{Lipschitz constants} of its elements:
\[\Lip(f)
\coloneqq \sup\left\{\frac{|f(x)-f(y)|}{d(x,y)}\colon x,y\in M,\ x\neq y\right\}.\]
This defines the \emph{Lipschitz space} of $(M,d,0)$ and its \emph{Lipschitz-free space} is given by 
\[\F(M)=\overline{\mathrm{span}}\left\{\delta(x)\in \Lip_0(M)^*\colon x\in M\right\}\sbs \Lip_0(M)^*,\]
where each $\delta(x)$ denotes the \emph{point evaluation functional} in $\Lip_0(M)^*$ given by 
\[\langle f,\delta(x)\rangle\coloneqq  f(x)\ \text{ for all }\ f\in\Lip_0(M).\]
Throughout, we use the notation $\langle f,\varphi\rangle$ to denote 
a functional $\varphi$ acting on a vector $f$.

The realization of the members of $\mathcal{F}(M)$ as functionals depends on the choice of the base point $0$. However, as a Banach space, $\mathcal{F}(M)$ is well-defined without reference to the base point (cf. \S\ref{change of base point and subspace identification subsection} or \cite[Chapter~3]{Weaver}).
In this context, note, that if $A\sbs M$, then $\F(A)$ can be canonically identified with a closed subspace of $\F(M)$ via
\[\F(A) = \overline{\mathrm{span}}\{\delta(x)-\delta(y)\in \F(M)\colon x,y\in A\}.\]
Whenever $0\in A$, since $\delta(0)= 0$, it follows that $\F(A) = \overline{\mathrm{span}}\{\delta(x)\in \F(M)\colon x\in A\}$.

A canonical reference is the textbook of Weaver \cite{Weaver}, in which Lipschitz free spaces are referred as Arens-Eells spaces, which is another common name for these spaces (among many others). Early influential work on free spaces from the Banach space perspective is due to Godefroy-Kalton \cite{GK}. Since then, Lipschitz-free spaces have become one of the main topics in Banach space theory; for some of the rapidly growing literature in the field, see \cite{NaorSChechtman2007SIAM,Godard2010PAMS,LancienPernecka2013JFA,CuthDouchaWojtaszczyk2016PAMS,AlbiacAnsorenaCuthDoucha2021TAMS,AliagaGartlandPetitjeanProchazka2022TAMS,Smith2026PLMS}. For a survey on the topic, we refer to  \cite{Godefroy2015CommMathSurvey}.

\subsection{Our main results} \label{ss:results}
Our main result reads as follows. 

\begin{thmA}[Proved as Theorem \ref{Thm.NonLEmb.Banach.dim.2} in \S\ref{SectionnonLembBanach2}] \label{thmA}
Let $X$ be a Banach space with $\dim (X) \geq 2$. Then $\F(X)$ is not $L$-embedded in $\F(X)^{**}$.
\end{thmA}

Notice that, if $\dim(X)$ is either $0$ or $1$, then $\F(X)$ is $L$-embedded in $\F(X)^{**}$. Indeed, if $\dim(X)=0$ then $\F(X)=\{0\}$ and, if $\dim(X)=1$, then $\F(X)\equiv L_1(\R)$. Thus, Theorem~\ref{Thm.NonLEmb.Banach.dim.2} is in fact a biconditional statement. 

We now provide a brief overview of the methods used to prove Theorem \ref{Thm.NonLEmb.Banach.dim.2}. The proof is by contradiction. We assume that $\dim(X) \geq 2$ and that $\F(X)$ is $L$-embedded in $\F(X)^{**}$.  We eventually show that this assumption implies that for any vectors $x,y$ in the unit sphere of $X$ with $\mathrm{dist}(y,\mathbb Rx)=1$, the sequence $(x,y)$ is isometrically equivalent to the unit vector basis of $\ell_1^2$ (see Lemma \ref{lem:l1basis}). It is not hard to see that such a rigid property is impossible for all such pairs. A quantitative refinement of this rigidity argument shows that a low-distortion containment of the unit ball of a normed space of dimension at least 2 in a metric space $M$ precludes the $L$-embeddability of $\mathcal{F}(M)$ in $\F(M)^{**}$ (see Thereom \ref{low-distortion.Thm.NonLEmb.Banach.dim.2}). In particular, reversible Finsler manifolds (e.g., Riemannian manifolds) of dimension at least 2 have non-$L$-embedded Lipschitz-free spaces (Corollary \ref{finsler.lipschitz.free.space.non.L-embeddability}).

The key ingredient for proving the (impossible) rigidity property described above is Theorem~\ref{thm:complete_compact}, which states that when $M$ is a complete pointed metric space, no $L$-projection $P\colon \F(M)^{**}\to \F(M)$ expands the support of any compactly supported functional. For measures, such a statement is quite intuitive: if $\mu=\mu_1+\mu_2$, where $\mu,\mu_1,\mu_2$ are Radon measures, and if the support of $\mu_1$ is not contained in the support of $\mu$, then the exceeding part must be cancelled out by $\mu_2$, which suggests that $\mu_1, \mu_2$ are not norm-minimal and the decomposition does not form an $\ell_1$-sum. However, not all functionals in $\F(M)^{**}$ are induced by measures on $M$ (or even on a compactification of $M$), see \cite[Corollary 6.3]{AP_measures}. In order to bring measures and their supports into the picture, we use the so-called \emph{De Leeuw measure representation} of functionals in $\F(M)^{**}$. This is one of the primary tools used to prove Theorem~\ref{thm:complete_compact}.
Let us outline the main steps in the proof.

We begin by discussing the meaning of ``supports'' of functionals in $\F(M)^{**}$. Given a complete pointed metric space $(M,d,0)$ and $A\sbs M$, set 
\begin{equation}\label{Defi.MathFraksmA}
\begin{split}
\mathfrak{S}(A) \coloneqq \big\{\varphi\in \F(M)^{**}: &\;\langle f,\varphi\rangle=0\textup{ for all }f\in\Lip_{0}(M)\\&\text{ vanishing on some neighborhood of }A \big\}.
\end{split}
\end{equation}
and
\begin{equation}\label{Defi.MathFrakDA}
\begin{split}
\mathfrak{D}(A) \coloneqq \big\{\varphi\in \F(M)^{**}: &\;\langle f,\varphi\rangle=0\textup{ for all }f\in\Lip_{0}(M)\\&\text{ constant on some neighborhood of }A \big\}.
\end{split}
\end{equation}
Heuristically, $\mathfrak{S}(A)$ should be seen as the elements of $\F(M)^{**}$ which are ``supported'' in $A$, and $\mathfrak{D}(A)$ is its base-point-free version. The \emph{(extended) support} of a functional $\varphi\in \F(M)^{**}$ was defined in \cite[Definition 3.5]{AP_measures} as $\mathcal{S}(\varphi):=\bigcap\{\overline{A}^{uM}:\varphi\in\mathfrak{S}(A)\}\sbs uM$,
where $uM$ denotes the {\em uniform compactification} of $M$, also called {\em Samuel  compactification}, whose definition will be recalled in \S\ref{Subsection.Compactifications}. We will work with $\mathcal{S}(\varphi)$ only in Proposition~\ref{prop:intersection} -- it is elsewhere more convenient for us to work with a subset $A$ of $M$ and the induced sets of functionals $\mathfrak{S}(A)$ and $\mathfrak{D}(A)$. Clearly, $\mathfrak{S}(A)$ and $\mathfrak{D}(A)$ are weak$^*$-closed subspaces  of $\F(M)^{**}$ containing $\F(A\cup\{0\})$ and $\F(A)$, respectively.

We now recall the De Leeuw transform. For a given complete pointed metric space $(M,d,0)$, we write
\[\widetilde{M} \coloneqq \left\{(x,y)\in M\times M\colon\, x\neq y\right\}.\]
For $f\in \Lip_0(M)$, let $\Phi f\colon \widetilde M\to \R$ be the continuous bounded function defined by 
\[\Phi f(x,y) \coloneqq \frac{f(x)-f(y)}{d(x,y)}\ \text{ for all }\ (x,y)\in \widetilde M.\]
We identify   $\Phi f $ with its (unique) extension to the Stone--\v{C}ech compactification $\beta \widetilde M$. Under this identification, the \emph{De Leeuw transform} is the linear isometric map 
\begin{align*}\Phi:\Lip_0(M)&\to C(\beta\widetilde{M})\qquad
f\mapsto \Phi f.
\end{align*} 

We let $\mathcal M(\beta \widetilde M)$ denote the space of Radon measures on $\beta \widetilde M$ endowed with its total variation norm.\footnote{Here a Radon measure will always mean a finite regular signed Borel measure.} We identify $\mathcal M(\beta \widetilde M)$ with the dual $C(\beta \widetilde M)^*$ in the usual way: elements in $\mathcal M(\beta \widetilde M)$ act on  $C(\beta \widetilde M)$ via integration. Since the adjoint $\Phi^*$ of the De Leeuw transform is surjective, every $\varphi\in \F(M)^{**}$ can be represented as $\Phi^*(\mu)$ for some $\mu\in \M(\beta\tilde M)$, but more can be said:
Define for $\varphi\in\F(M)^{**}$ the set of {\it optimal representations} by
\begin{equation} \label{DefiMopVarphi}
\mathcal{M}_{op}(\varphi) \coloneqq \left\{\mu\!\in\! \mathcal{M}(\beta\widetilde{M})\colon\, \mu\!\geq\!0, \|\mu\|\!=\!\|\varphi\|, \langle f,\varphi\rangle=\int_{\beta\widetilde{M}}\Phi f \,d\mu, \forall f\in\Lip_0(M)\right\}.
\end{equation}
By \cite[Proposition 3]{Aliaga}, the set $\mathcal{M}_{op}(\varphi)$ is non-empty for every $\varphi\in\F(M)^{**}$. Thus, every $\varphi\in \F(M)^{**}$ can be represented as
$\Phi^*(\mu)$ for some positive $\mu\in \M(\beta\tilde M)$ with $\|\varphi\|=\|\mu\|$.
A measure in $\mathcal{M}_{op}(\varphi)$ should be seen as a positive coding for the functional $\varphi$.
 
 Our study focuses on the relation between the support of a functional $\varphi\in \F(M)^{**}$ and the support of its representations $\mu\in \mathcal{M}_{op}(\varphi)$.
The main technical result in this direction is Theorem \ref{prop:shadow_almost_norm_attaining}. It states that, for closed $A\sbs M$ and for certain ``nice'' functionals  $\varphi\in \mathfrak{D}(A)$ (of which there are sufficiently many), the support of any $\mu\in\M_{op}(\varphi)$ is ``contained'' in the geodesic extension of $A$. For a compact set $A\sbs M$, the {\it geodesic extension} is the set  
\[\left\{z\in M\colon \exists x,y\in A\ \text{ with }\ d(x,y)=d(x,z)+d(z,y)\right\}\]
(see \S\ref{Subsection.Gex} for the definition for a general subset $A$). The two considered supports -- that of a functional and that of its De Leeuw representing measure -- are subsets of different spaces. Therefore, we consider coordinate projections from $\beta \widetilde M$ to the uniform compactification $uM$ of $M$. This is an appropriate compactification to work within, as Lipschitz functions separate closed subsets of $uM$ (see \S\ref{Subsection.Compactifications} for details). 

On the other hand, it is known by \cite[Lemma 8]{Aliaga} that the support of a weak$^*$ continuous functional is contained in the support of its representing measure. This, together with Theorem \ref{prop:shadow_almost_norm_attaining}, allows us to move the argument to the setting of De Leeuw representing measures. The main advantage is that the non-expansion of supports under $L$-projections can be deduced easily in the case of representing measures (see Proposition \ref{prop:support_preservation_deLeeuw}). Combining all the above, we readily get a weaker version of the support-preservation result: $L$-projections map elements supported in $A$ to elements supported in the geodesic extension of $A$ (see Corollary \ref{cor:L_projection_shadow}).

Lastly, if $A$ is a subset that equals its geodesic extension, i.e., $A$ is \emph{geodesically closed}, then $L$-projections map elements supported in $A$ to elements supported in $A$ (see Theorem \ref{thm:L_projection_supports}). However, geodesic closedness often fails to hold; even compact line segments like $[0,1]x \sbs \ell_1$ may fail to be geodesically closed. This challenge is overcome in Theorem~\ref{thm:complete_compact}, where we decompose a functional into ones subordinated to a sufficiently fine cover of its compact support.

The path to prove Theorem \ref{Thm.NonLEmb.Banach.dim.2} described above has other consequences of potential interest. In particular, the methods are especially suitable for proving certain hereditary results on $L$-embeddability beyond the Banach space scenario. These findings are summarized in the next theorem.

\begin{thmB}[Proved as Theorems \ref{Thm.LembPassesAcompact} and \ref{thm:L_embed_hereditary_geodesic} in \S\ref{SectionLProjHer}]
Let $(M,d,0)$ be a complete pointed metric space and let $A\sbs M$. Let $P:\mathcal{F} (M)^{**}\to \mathcal{F}(M)$ be an $L$-projection. If $A$ is compact or if $A$ is geodesically closed, then $P(\mathcal{F}(A)^{**}) \sbs \mathcal{F}(A)$. 

In particular, under either of these conditions, if $\mathcal{F}(M)$ is $L$-embedded in $\F(M)^{**}$, then $\mathcal{F}(A)$ is $L$-embedded in $\F(A)^{**}$.
\end{thmB}

\section{Preliminaries}
In this preliminary section, we review some concepts and terminology needed for the article, in addition to proving a few basic statements (Lemma~\ref{lem:gap function vs approximate intervals}, Proposition~\ref{prop:decomposition_S(A)} and Proposition~\ref{prop:sufficient_for_assumption}) used later.

We start by fixing some simple notation. Since we work with Lipschitz-free spaces, we will consider  {complete} \emph{pointed metric spaces}, i.e.,  {complete} metric spaces $(M,d)$ with a distinguished base point, which we typically denote by $0\in M$. We write $(M,d,0)$ for a pointed metric space. For a given $A\sbs M$ and $\varepsilon>0$,  we write \[B(A,\varepsilon)\coloneqq\{x\in M: d(x, A)< \varepsilon\}.\] For simplicity, if $A$ is a singleton, say $A=\{y\}$, we simply write   $B(y,\varepsilon)$ instead of $B(\{y\},\varepsilon)$.

\subsection{Compactifications}\label{Subsection.Compactifications}
Throughout these notes, we will make heavy use of two compactifications for metric spaces: the \emph{Stone--\v{C}ech compactification} and the  {\emph{uniform compactification}, also known as the \emph{Samuel compactification}}. We now quickly recall their definitions and main properties -- we  restrict our exposition to metric spaces since this will be our setting.

Given a metric space $(M,d)$, we denote its \emph{Stone--\v{C}ech compactification} by $\beta M$. The compactification $\beta M$ is a compact Hausdorff space containing  $M$ homeomorphically as a dense subset and it is characterized by the following extension property: if $K$ is a compact Hausdorff space and $f\colon M\to K$ is a continuous function, then there exists a unique continuous extension  of $f$ to the whole $\beta M$. We refer the reader to \cite[Chapter 5]{MunkresBookTopology} for a detailed treatment of the subject.

Given a metric space $(M,d)$, we denote its  {\emph{uniform compactification}, also known as its \emph{Samuel compactification}, by $uM$}. The compactification $uM$ is a compact Hausdorff space containing $M$ homeomorphically as a dense subset and it is characterized by the following properties (see \cite[Corollary 2.4]{Woods}):
\begin{enumerate}
\item Every bounded uniformly continuous function $f\colon M\to \R$ can be uniquely extended to   a continuous function $uM\to\R$.
\item Given  $A,B\sbs M$, their closures  in $uM$ are disjoint if and only if $d(A,B)>0$.
\end{enumerate}
We refer the reader to \cite{Woods} for the basics of  {uniform} compactification and to  \cite{Weaver,AP_measures} for its use in the context of Lipschitz-free spaces. A concrete way to define $uM$ (up to homeomorphisms) is to let $U_b(M)$ be the space of bounded, uniformly continuous $\R$-valued functions on $M$, and then to define $uM$ as the closure of the range of 
the map 
$$\iota: M\to \R^{U_b(M)},\quad x\mapsto (f(x))_{f\in U_b(M)}.$$
Using this definition, we can easily deduce the following two simple properties satisfied by this compactification, which we will tacitly use throughout these notes (see, e.g., \cite[Propositions 2.8 and 2.9]{AP_measures}).

\begin{proposition} Let $(M,d)$ be a metric space.
    \begin{enumerate}
\item Every Lipschitz function $f\colon M\to \R$ can be uniquely extended to   a continuous function   $uM\to [-\infty,\infty]$.
\item Given disjoint closed subsets $A,B\sbs uM$, there exist open sets $U,V\sbs uM$ containing $A$ and $B$, respectively, such that $d(U\cap M,V\cap M)>0$.
\end{enumerate}\label{Lemma.AP_measures.sM.Basics}
\end{proposition}
 
\begin{convention}\label{Convention1} Hereinafter, any real-valued Lipschitz function $f$ defined on a metric space $M$ will be considered to be continuously extended to $uM$, and we still denote this extension by $f$. We emphasize that extensions may no longer be real-valued but instead take values in the extended real line $[-\infty,\infty]$.
\end{convention}

We finish this subsection by singling out an open and dense subset  of $uM$ which will play an important role herein. Precisely, we define, after extending the map $d(0,\cdot): M\to \R$ to $uM$, 
\[rM \coloneqq \left\{\xi\in uM\,\colon\, d(\xi,0)<\infty\right\}.\] 
Clearly, 
\[M\sbs rM\sbs uM.\]

\begin{example} \label{Example.srM.proper}
We have that $rM=M$ if and only if the metric space $M$ is proper, and $rM=uM$ if and only if $M$ is bounded. 
\end{example}
 
\subsection{Geodesic extensions}
\label{Subsection.Gex}
Let $(M,d)$ be a metric space, $x,y\in M$, and $\eps>0$. We write  
\[[x,y]_{\varepsilon} \coloneqq \left\{z\in M\colon\,d(x,y)>d(x,z)+d(z,y)-\varepsilon\right\}.\] 
So, $[x,y]_\eps$ should be seen as an ``$\eps$-approximation'' of the \emph{metric segment}
\[[x,y] \coloneqq \left\{z\in M\colon\,d(x,y)=d(x,z)+d(z,y)\right\}.\]
Given a subset 
$A\sbs M$, we define 
\[\gex_{\varepsilon}(A) \coloneqq \bigcup\left\{[x,y]_{\varepsilon}\colon\,x,y\in A\right\}.\]

\begin{definition}\label{Defi.Geo.Close}
Given a metric space  {$(M,d)$} and  $A\sbs M$, the \emph{geodesic extension of $A$} is defined as 
\[\gex(A) \coloneqq \bigcap_{\varepsilon>0}\overline{\gex_{\varepsilon}(A)}\]
and we say that $A$ is \emph{geodesically closed} if $A=\gex(A)$.
\end{definition}

Clearly, $\gex(A)$ is a closed set and, when $A$ is  compact, we have that
\begin{equation}
    \label{eq:gex_compact}
    \gex(A)=\bigcup\{[x,y]\colon\,x,y\in A\}.
\end{equation}
We will also need a larger version of $\gex(A)$ where the closures in its definition are taken in the  {uniform} compactification. Precisely, given a metric space $M$ and $A\sbs M$, we let 
\[\gex^{u}(A) \coloneqq \bigcap_{\varepsilon>0}\overline{\gex_{\varepsilon}(A)}^{uM}.\] It is easy to verify that \[\gex(A)=\gex^{u}(A)\cap M.\]

\subsection{McShane extensions and the McShane gap function} \label{SubsectionMcShane}
It is well known that real-valued Lipschitz functions have Lipschitz extensions which do not increase their Lipschitz constants. They are however by no means unique. McShane extensions provide the largest and the smallest possible; 
we review these constructions in this subsection. We shall also need to quantify the gap between these extensions and relate them to geodesic extensions.

\begin{definition}
    Let $(M,d)$ be a metric space, $A\sbs B\sbs M$ with $A\neq \emptyset$, and $f\colon B\to \R$ be a $1$-Lipschitz function. 
    \begin{enumerate}
        \item The \emph{McShane sup-convolution} $\underline{f}_{A}\colon\, M\to\R$ is the extension of $f\vert_A\colon A\to \R$ to $M$ given by 
\[\underline{f}_{A}(x) \coloneqq \sup\left\{f(y)-d(x,y)\colon y\in A\right\}\]
for all $x\in M$.
   \item The \emph{McShane inf-convolution} $\overline{f}_{A}\colon\, M\to\R$ is the extension of $f\vert_A\colon A\to \R$ to $M$ given by 
\[\overline{f}_{A}(x) \coloneqq \inf\left\{f(y)+d(x,y)\colon y\in A\right\}\]
for all $x\in M$.
    \end{enumerate}
\end{definition}

The McShane convolutions are $1$-Lipschitz, extend the restriction of $f$ to $A$, and satisfy $\underline{f}_A\leq\overline{f}_A$ on $M$. Hence, for any $\alpha\in [0,1]$, the function  $\alpha \underline{f}_{A}+(1-\alpha)\overline{f}_{A}:M\to\R$  also coincides with $f$ on $A$ and is $1$-Lipschitz. Moreover, if $g\colon M\to \R$ is a $1$-Lipschitz extension of $f\vert_A$, then 
\[\underline{f}_A\leq g\leq \overline{f}_A.\]
Notice also that, if $\varepsilon>0$ is such that $ {\overline{B(A,\varepsilon)}}\sbs B$, then
\begin{equation}
\label{eq:extensions_from_neighbourhoods}
\sup_{x\in M}\left|\underline{f}_{A}(x)-\underline{f}_{ {\overline{B(A,\varepsilon)}}}(x)\right|\leq 2\varepsilon\quad\quad\textup{and}\quad\quad\sup_{x\in M}\left|\overline{f}_{A}(x)-\overline{f}_{ {\overline{B(A,\varepsilon)}}}(x)\right|\leq 2\varepsilon.
\end{equation}

\begin{definition}\label{Defi.Gap.Function.ga.}
Given a metric space  {$(M,d)$} and $A\sbs M$, the \emph{McShane gap function} $g_A:M\to[0,\infty)$ is defined by
\[g_A(x) \coloneqq  \inf\left\{\overline f_A(x) - \underline f_A(x):f:A\to\mathbb{R}\text{ is }1\text{-Lipschitz}\right\}\]
for all $x\in M$. 
\end{definition}

Clearly, the  McShane gap function $g_A$ is zero on $A$. We should also notice that $g_A$ is always $2$-Lipschitz (see \cite[Proposition 1.23]{Weaver}).

The following lemma is a ``global'' version of \cite[Lemma 5.1]{APS1} for the  {uniform} compactification $uM$. 
\begin{lemma}
\label{lem:gap function vs approximate intervals}
Let $(M,d)$ be a  metric space and $A$ be a non-empty subset of $M$.
Then 
\[ \gex^{u}(A)=\left\{\xi\in uM : g_A(\xi)=0\right\} .\]
Moreover, if $M$ is proper and $A$ is bounded, then
\[\gex(A)=\left\{\xi\in uM: g_A(\xi)=0\right\} .\]
\end{lemma}

\begin{proof}
We begin by proving a preparatory claim. 

\begin{claim}  \label{eq:gap_epsilon}
Given $\eps>0$, we have $\gex_{\varepsilon}(A)=\left\{x\in M: g_A(x)<\varepsilon\right\}$.
\end{claim}

\begin{proof}
Let $x\in \gex_{\varepsilon}(A)$ and let us show that $g_A(x)<\eps$. For that, fix $y,z\in A$ and $0<\varepsilon'<\varepsilon$ such that 
\begin{equation}\label{Eq.gap function vs approximate interval.1}
d(y,z)> d(x,y)+d(x,z) - \varepsilon' 
\end{equation} 
and consider the $1$-Lipschitz function $f:A\to\mathbb{R}$ given by \[f(w) = d(y,w)\ \text{ for all }\  w\in A.\] Let $h:M\to\mathbb{R}$ be an arbitrary $1$-Lipschitz extension of $f$. As $y\in A$, $h(y)=f(y)=0$ and we have  
\begin{equation}\label{Eq.gap function vs approximate interval.2}h(x)=h(x)-h(y)\leq d(x,y).\end{equation}

Let us notice that \begin{align}\label{Eq.gap function vs approximate interval.3}d(y,x)-\varepsilon' < h(x).\end{align} Suppose towards a contradiction this fails. Then, using \eqref{Eq.gap function vs approximate interval.1}, we have \begin{align*} 
h(z)&= h(z) - h(x)+h(x) \\
&\leq d(z,x)+ d(y,x) -\varepsilon'  \notag\\
&< d(y,z).\notag
\end{align*}
Since $z\in A$, $h(z)=d(y,z)$ and this gives us a contradiction. By \eqref{Eq.gap function vs approximate interval.2} and \eqref{Eq.gap function vs approximate interval.3},   $h(x)$ is in  $(d(x,y)-\varepsilon',d(x,y)]$. Since $h$ was an arbitrary extension of  $f$, we have \[g_A(x)\leq \overline{f}_A(x) - \underline{f}_A(x)\leq \varepsilon'<\varepsilon.\]

Suppose now  $x\in M$ is such that   $g_A(x) < \varepsilon$. Then, there exists a $1$-Lipschitz function $f:A\to\mathbb{R}$ such that
\[
\overline{f}_A(x) - \underline{f}_A(x) < \varepsilon.
\]
We may thus find $y,z\in A$ for which
\[f(z)-f(y) +d(x,z)+d(x,y)< \varepsilon,\] and therefore,
\[d(x,z)+d(x,y)< f(y)-f(z)+\varepsilon \leq d(y,z)+\varepsilon.\]
In other words, $x\in\gex_{\varepsilon}(A)$.
\end{proof}

We now prove the lemma. Firstly notice that, as $g_A$ is continuous on $uM$ and as $M$ is dense in $uM$, 
\[\{\xi\in uM : g_A(\xi)=0\}= \bigcap_{\varepsilon>0} \{\xi\in uM : g_A(\xi)<\varepsilon\}= \bigcap_{\varepsilon>0} \overline{\{x\in M: g_A(x)<\varepsilon\}}^{uM}.\]
Hence, by Claim~\ref{eq:gap_epsilon},   
\[\{\xi\in uM: g_A(\xi)=0\}=\bigcap_{\varepsilon>0} \overline{\gex_{\varepsilon}(A)}^{uM}=\gex^{u}(A).\]

At last, for the last statement, notice that if $A$ is bounded, then so is $\gex_{\eps}(A)$ for every $\eps>0$. Hence, assuming $M$ to be proper,  $\gex_{\eps}(A)$ is relatively compact and we have 
\[\gex(A)=\bigcap_{\eps>0}\overline{\gex_{\eps}(A)}=\bigcap_{\eps>0}\overline{\gex_{\eps}(A)}^{uM}=\{\xi\in uM: g_A(\xi)=0\}\] by the above.
\end{proof}

\subsection{Decomposition of functionals in $\F(M)^{**}$}
Let $(M,d,0)$ be a pointed metric space  and  $h:M\to \R$ be Lipschitz with   bounded support. This induces a weighting operator 
\begin{align}\label{Eq.Wh.formula}
   W_h\colon \Lip_0(M)&\to\Lip_0(M)\\ f&\mapsto fh .\notag
\end{align}
By \cite[Lemma 2.3]{APPP}  {(where $W_h$ is denoted $T_h$)}, this is a $w^*$-to-$w^*$ continuous linear operator with \[\|W_h\|\leq\|h\|_\infty+\rad(\supp(h))\Lip(h),\]
where for a given $A\sbs M$ we consider
\[\mathrm{rad}(A) \coloneqq \sup\{d(x,0)\colon x\in A\}\]
and
\[\supp(h):=\overline{\{x\in M: h(x)\neq 0\}}.\]
Under the canonical isomorphism between $\Lip_0(M)$ and $\F(M)^*$, we see the adjoint of $W_h$ as an operator 
\[W^*_h\colon \F(M)^{**}\to \F(M)^{**}.\]

For each $n\in\N$, let $\daleth_n:M\to\R$ be given by
\begin{equation}\label{eq:daleth}
\daleth_n(x)\coloneqq \begin{cases}
1 &\text{if } d(x,0) \leq 2^n, \\
2-2^{-n}d(x,0) &\text{if } 2^n \leq d(x,0) \leq 2^{n+1}, \\
0 &\text{if } 2^{n+1} \leq d(x,0) .
\end{cases}
\end{equation}
The following was introduced in \cite[Definition 3.1]{AP_measures}.

\begin{definition}
    Let $(M,d,0)$ be a complete pointed metric space and $(\daleth_n)_n$ be as above. Let $\varphi\in \F(M)^{**}$.
    \begin{enumerate}
        \item We say that $\varphi$ \emph{avoids infinity} if \[\varphi=\lim_{n\to\infty}W_{\daleth_n}^*(\varphi).\]
      \item We say that $\varphi$ is \emph{concentrated at infinity} if \[W_{\daleth_n}^*(\varphi)=0\ \text{   for every }\  n\in \N.\]
    \end{enumerate}
\end{definition}
 
  \begin{example}
       Every functional in $\F(M)$ avoids infinity (see \cite[Pages 18 and 19]{AP_measures}).
  \end{example}

  Notice that if $\varphi\in \F(M)^{**}$ is arbitrary, then $(W_{\daleth_n}^*(\varphi))_n$ is Cauchy (see \cite[Section 3.1]{AP_measures}). Therefore, this allows us to define 
\begin{equation}\label{Eq.Def.varphi_M}\varphi_M \coloneqq \lim_{n\to\infty}W_{\daleth_n}^*(\varphi)\ \text{  and }\  \varphi_{\infty} \coloneqq \varphi-\varphi_M.
  \end{equation}
Clearly, $\varphi_M$ avoids infinity and $\varphi_\infty$  
is concentrated at infinity. So, each functional in $\F(M)^{**}$ can be decomposed into a part avoiding infinity and one concentrated at it. Moreover, this decomposes $\F(M)^{**}$ into an $\ell_1$-sum. Precisely: 

\begin{proposition}
   \emph{(}\cite[Corollary 3.4]{AP_measures}\emph{).}
 Let $(M,d,0)$ be a complete pointed metric space and $\varphi\in \F(M)^{**}$.
 Then 
 \[\|\varphi\|=\|\varphi_M\|+\|\varphi_\infty\|.\]\label{Prop.AP_measures.Cor.3.4}
\end{proposition}

\subsection{Change of base point and subspace identification}\label{change of base point and subspace identification subsection}

As a Banach space, the Lipschitz free space of a metric space $M$ does not depend on the choice of the base point. However, the standard definition of the support of a functional $\varphi\in\mathcal{F}(M)^{**}$ depends on the base point. Our proofs rely essentially on handling notions of support of functionals and, on the other hand, necessitate changing the base point in certain instances. We formalize the change of base point maps and record, among other information, how to canonically identify, for an arbitrary $A\sbs M$, the space $\mathcal{F}(A)$ with a subspace of $\F(M)$.

\begin{definition}
\label{notation change of base point}
Let $(M,d,0)$ be pointed metric space and let $\mathrm{Lip}(M)$ denote the set of all Lipschitz functions $f:M\to \R$. Consider the base point selection mapping
\begin{align*}
T_0:\Lip(M)&\to\Lip_0(M)\text{ given by } f\mapsto f-f(0)\mathbf{1}_M,
\end{align*} 
where $\mathbf{1}_M$ denotes the function on $M$ constantly equal to $1$.
\end{definition}

The items in the following statement are straightforward.

\begin{proposition}
\label{change of base point proposition}
Consider a pointed metric space $(M,d,0)$ and $a\in M$. Let $\F_a(M)$ denote the Lipschitz free space of $(M,d,a)$. The following hold.
\begin{enumerate}
    
    \item\label{change of base point proposition i} The operator \[T_{0,a} = T_0|_{\mathrm{Lip}_a(M)}:\mathrm{Lip}_a(M)\to \mathrm{Lip}_0(M)\]
    is an onto linear isometry with inverse $T_{a,0}:\mathrm{Lip}_0(M)\to\mathrm{Lip}_a(M)$. Therefore,
    \[T_{0,a}^*:\F(M)^{**}\to \F_a(M)^{**}\]
    is an onto linear isometry with inverse $T_{a,0}^*:\F(M)^{**}\to \F_a(M)^{**}$.
    
    \item\label{change of base point proposition ii} For all $x,y\in M$,
    \[T_{0,a}^*(\delta(x) - \delta(y))=\delta(x)-\delta(y).\]
    In particular, $T_{0,a}^*(\F(M)) = \F_a(M)$, and $T_{0,a}$ is a $w^*$-to-$w^*$ homeomorphism.
    
\end{enumerate}
\end{proposition}

\begin{remark}
\label{change of base point remarks}
Let $(M,d,0)$ be a pointed metric space.
\begin{enumerate}

\item\label{change of base point remarks 0} By Proposition \ref{change of base point proposition} \eqref{change of base point proposition ii}, $\mathcal{F}(M)$ is isometric to $\mathcal{F}_a(M)$, for any $a\in M$. Therefore, as a Banach space, $\mathcal{F}(M)$ is well-defined without reference to a base point.

\item\label{change of base point remarks ii} The De Leeuw transform is invariant under change of base points, that is, for all $a\in M$, $\Phi = \Phi T_{0,a}$. In particular, for all $\varphi\in\F(M)^{**}$, $ \mathcal{M}_{op}(\varphi) = \mathcal{M}_{op}(T^*_{0,a}\varphi)$ (see \eqref{DefiMopVarphi} in the introduction the definition of $\mathcal{M}_{op}(\varphi)$.

\item\label{change of base point remarks iii} By the above item and \cite[Corollary 2.12]{APS2}, $\varphi\in\F(M)^{**}$ avoids infinity if and only if $T_{0,a}^*\varphi\in\F_a(M)^{**}$ does.

\end{enumerate}
\end{remark}

\begin{definition}
For a pointed metric space $(M,d,0)$, $A\sbs M$, and $a\in A$, we define the change-of-base-point-plus-restriction map
\begin{align*}
    R:\Lip_0(M)\to \Lip_a(A)\text{ given by }f\mapsto (T_{a,0}f)\vert_A.
\end{align*}
\end{definition}

\begin{proposition}
\label{canonical identification}
Let $(M,d,0)$ be a pointed metric space, $A\sbs M$, and $a\in A$.
\begin{enumerate}
    
    \item The operator $R:\Lip_0(M)\to \Lip_a(A)$ maps the closed unit ball of $\Lip_0(M)$ onto the closed unit ball of $\Lip_a(A)$. In particular, the operator $R^*:\F_a(A)^{**}\to \F(M)^{**}$ is an isometric embedding preserving the norm-attainment of functionals.
    
    \item For all $x,y\in A$, \[R^*(\delta(x)-\delta(y)) = \delta(x)-\delta(y).\]
    In particular, $R^*(\F_a(A)) = \overline{\mathrm{span}}\{\delta(x)-\delta(y):x,y\in A\}\sbs \F(M)$, and $R$ is w$^*$-to-w$^*$ continuous.

\end{enumerate}
\end{proposition}

\begin{proof}
   For the first item, we need that, using McShane's extension, Lipschitz functions $\in\mathrm{Lip}_a(A)$ have extensions to $\mathrm{Lip}_a(M)$ of the same norm, which are then mapped via $T_{0,a}$ to $\mathrm{Lip}_0(M)$. The second item is straightforward.
\end{proof}

\begin{remark}
Henceforth, we identify $\F(A)^{**}$ with the subspace  $R^*(\F(A)^{**}) $ of $\F(M)^{**}$ and, correspondingly, $\mathcal{F}(A)$ with the subspace $\overline{\mathrm{span}}\{\delta(x)-\delta(y):x,y\in A\}$ of $\mathcal{F}(M)$.
\end{remark}

\subsection{Functionals in $\F(M)^{**}$ supported in $A\sbs M$} 

It is of fundamental importance for us to understand when elements of $\F(M)^{**}$ are in $\F(M)$.  We gather two results of \cite{APPP} below which allow us to decide this for elements in $\mathfrak{S}(A)$ and which will also be used repeatedly below.

\begin{theorem}\emph{(}\cite[Theorem 2.1]{APPP}\emph{).}
Let $(M,d,0)$ be a complete pointed metric space and $(K_i)_{i\in I}$ be a family of closed subsets of $M$. Then, 
    \[\F\left(\bigcap_{i\in I}K_i\cup\{0\}\right)=\bigcap_{i\in I}\F\left( K_i\cup\{0\}\right).\]\label{Theorem.APPP.2.1}
\end{theorem}

Theorem \ref{Theorem.APPP.2.1} allows us to give a coherent definition for the support of an element of $\F(M)$. Precisely, for $m\in \F(M)$, we let 
\[\supp(m)\coloneqq\bigcap\left\{K\sbs M: K\text{ is non-empty and closed and } m\in \F(K\cup\{0\})\right\}.\]
So, $m\in \F(\supp(m)\cup\{0\})$. In fact, $\supp(m)$ is the smallest non-empty closed subset $K$ of $M$ such that $m\in\F(K\cup\{0\})$. 

Here is the other result of \cite{APPP} which we will need.

\begin{proposition}\emph{(}\cite[Proposition 2.6]{APPP}\emph{).}
Let $(M,d,0)$ be a complete pointed metric space, $m\in \F(M)$,  and $K\sbs M$ be closed. Then  {the following are equivalent:
\begin{enumerate}
\item $m\in \F(K\cup\{0\})$,
\item $\supp(m)\sbs K$,
\item $\langle f,m\rangle=0$ whenever $f\in\Lip_0(M)$ satisfies $f\vert_K=0$.
\end{enumerate}}
\label{prop.APPP.2.6}
\end{proposition}

By Theorem \ref{Theorem.APPP.2.1} and Proposition \ref{prop.APPP.2.6},  we have   \begin{equation}
\label{eq:S(A)_in_F(M)}
\begin{split}
\mathfrak{S}(A)\cap\F(M) &
=\bigcap\{\F(K\cup\{0\} ): K\sbs M\textup{ is a closed neighborhood of }A\}\\
&=\F(A \cup\{0\}),
\end{split}
\end{equation}
where $\mathfrak{S}(A)$ and $\mathfrak{D}(A)$ were defined in \eqref{Defi.MathFraksmA} and \eqref{Defi.MathFrakDA}. Clearly, if $0\in A$, then $\mathfrak D(A)=\mathfrak S(A)$. If $0\not\in A$, the next proposition shows that $\mathfrak D(A)$ is a hyperplane of $\mathfrak S(A)$.

\begin{proposition}
\label{prop:decomposition_S(A)} 
Let $(M,d,0)$ be a complete pointed metric space and $A\sbs M$ be a non-empty closed subset with $0\not\in A$. Then, for all $x\in A$, we have 
\[\mathfrak{S}(A)=\R\delta(x)\oplus\mathfrak{D}(A).\]
\end{proposition}

\begin{proof}
Fix an open $U\sbs M$ containing $A$ and such that $0\notin \overline{U}$. Let     $f\in\Lip_0(M )$ be  such that  $f|_U = 1$. 

Given $x\in A$, let us show that  \[\varphi - \langle f,\varphi\rangle\delta(x)\in\mathfrak{D}(A) \ \text{ for all }\ \varphi \in\mathfrak{S}(A).\] 
Indeed, let $\varphi \in\mathfrak{S}(A)$ and pick a  $g\in\Lip_0(M)$ which is  constant on some neighborhood $V\sbs M$ of $A$. Then, $g - g(x)f=0$ on the neighborhood $U\cap V$ of $A$. As $\varphi\in \mathfrak{S}(A)$, this gives
\[  \langle g, \varphi - \langle f,\varphi\rangle\delta(x)\rangle= \langle g - g(x)f,\varphi\rangle=0.\]
By the arbitrariness of $g$, this shows that $ \varphi - \langle f,\varphi\rangle\delta(x)\in \mathfrak D(A)$ as desired. 

At last, let  $P:\mathfrak{S}(A)\to\R\delta(x)$ be given by \[P(\varphi)=\langle f,\varphi\rangle\delta(x)\ \text{ for all }\ \varphi\in \mathfrak S(A).\]
So, $P$ is a bounded linear projection. If $z\in A$, then, since  $\F(A\cup\{0\})\sbs \mathfrak S(A)$, $\delta(z)\in \mathfrak S(A)$ and \[P(\delta(z))=f(z)\delta(x)=\delta(x)\neq 0.\] So, $P$ is surjective. Finally, notice that  $\ker(P)=\mathfrak{D}(A)$. Indeed, the inclusion ``$\sps$'' is immediate. For the other inclusion, let   $\varphi\in \ker(P)$ and let  $g\in \Lip_0(M)$ be constant on a neighborhood of $A$.  Set $t=g(x)$ and notice that $tf-g$ is zero on a neighborhood of $A$. As  $\varphi\in \mathfrak S(A)$, $\langle tf-g,\varphi\rangle=0$. Therefore, since  $\varphi\in \ker(P)$, $\langle f,\varphi\rangle=0$ and we conclude that   $\langle g,\varphi\rangle=0$. So, $\varphi\in \mathfrak D(A)$.
\end{proof}

\begin{proposition}
\label{how support sets interact with change of base point}
Let $(M,d,0)$ be a pointed metric space and $a\in M$. Let $\F_a(M)$ denote the Lipschitz free space of $(M,d,a)$ and define the subsets $\mathfrak{S}_a(A)$ and $\mathfrak{D}_a(A)$ of $\F_a(M)^{**}$ in the obvious way. Then, for all $A\sbs M$,
\[T^*_{0,a}(\mathfrak{D}(A)) = \mathfrak{D}_a(A)\]
In particular, if $a\in A$, then $T_{0,a}^*(\mathfrak{D}(A)) = \mathfrak{S}_a(A)$.
\end{proposition}

\begin{proof}
Note that, for a subset $U$ of $A$, the set of all $f\in \mathrm{Lip}_a(M)$ that are constant on $U$ is mapped under $T_{0,a}$ to the set of all $f\in\mathrm{Lip}_0(M)$ that are constant on $U$, and for the second part, note that $\mathfrak{S}(A) = \mathfrak{D}(A)$ whenever $0\in A$.
\end{proof}

\begin{remark}
Note that the definition of the set $\mathfrak{S}(A)$ is dependent on the choice of the base point in $M$, while the definition of the set $\mathfrak{D}(A)$ is not.    
\end{remark}

\begin{remark}
\label{subset free space bidual containement in constant vanishers}
Let $M$ be a pointed metric space, $A\sbs M$ and $a\in A$. Then,
\begin{equation*}
\F(A)^{**}\sbs \mathfrak{D}(A),
\end{equation*}
where we canonically identify $\F(A)^{**} = R^*(\F(A)^{**})$, as in Proposition \ref{canonical identification}. This is because, clearly, $\F(A)\sbs\mathfrak{D}(A)$ and $\mathfrak{D}(A)$ is $w^*$-closed.
\end{remark}

\subsection{A crucial assumption on $\varphi\in\mathfrak{S}(A)$}
In order to simplify our statements below, we isolate an assumption which will be recurrent in most of our main technical results. 

\begin{assumption}
\label{Assumption_Pavlos}
Let $(M,d,0)$ be a complete pointed metric space, $A\sbs M$ and $\varphi\in\mathfrak{D}(A)$. Assume that there exists $(h_n)_{n=1}^\infty\sbs B_{\Lip_0(M)}$ such that 
\begin{enumerate}[label=\textup{(\roman*)}]
\item
\label{assumption_i}
$\|\varphi\|=\lim_n\langle h_n,\varphi\rangle$ and 
\item
\label{assumption_ii}
$({\overline{h_n}}_{A})_n$ and $({\underline{h_n}}_{A})_n$  are equicontinuous on $rM$. 
\end{enumerate}
\end{assumption}

\begin{remark}
Assumption \ref{Assumption_Pavlos} is invariant under change of base point.  That is, for any $a\in M$, the functional $\varphi\in\mathfrak{D}(A)$ satisfies Assumption \ref{Assumption_Pavlos}, witnessed by a sequence $(h_n)_{n=1}^\infty$, if and only if the functional $T_{0,a}^*\varphi\in\mathfrak{D}_a(A)$ also satisfies Assumption \ref{Assumption_Pavlos}, witnessed by the sequence $(T_{a,0}h_n)_{n=1}^\infty$. This is because, for all $f\in\mathrm{Lip}_0(M)$,
\[T_{a,0}(\overline{f}_A) = \overline{T_{a,0}f}_A\text{ and }T_{a,0}(\underline{f}_A) = \underline{T_{a,0}f}_A.\]
\end{remark}

We say that a sequence $(h_n)_n$ as in Assumption~\ref{Assumption_Pavlos}\ref{assumption_i} is \emph{norming for $\varphi$}. If a constant sequence can be taken, i.e., if there is a single 
 $h\in B_{\Lip_0(M)}$ with $\|\varphi\|=\langle h,\varphi\rangle$, we simply say that $h$ is \emph{norming for $\varphi$} and that $\varphi$ is \emph{norm attaining} or that it \emph{attains its norm}. For clarity and convenience, we unfold Assumption~\ref{Assumption_Pavlos}\ref{assumption_ii}: this condition   is readily seen to be equivalent to the statement that for every $\xi\in rM$ and every $\varepsilon>0$ there exists  an open set $U\sbs uM$ containing $\xi$ such that  for every $x\in U\cap M$ and $n\in \N$, we have \[\left|{\overline{h_n}}_{A}(\xi)-{\overline{h_n}}_{A}(x)\right|<\varepsilon\quad\textup{ and }\quad\left|{\underline{h_n}}_{A}(\xi)-{\underline{h_n}}_{A}(x)\right|<\varepsilon.\]

The next proposition gives some sufficient conditions so that Assumption~\ref{Assumption_Pavlos} holds.

\begin{proposition}
\label{prop:sufficient_for_assumption}
Let $(M,d,0)$ be a complete pointed metric space, $A\sbs M$ be closed and $\varphi\in\mathfrak{D}(A)$.  Assumption~\ref{Assumption_Pavlos} is satisfied whenever at least one of the following conditions holds:
\begin{enumerate}
\item\label{prop:sufficient_for_assumption.Item1}
 $\varphi$ is norm attaining,
\item\label{prop:sufficient_for_assumption.Item2} $M$ is proper,
\item\label{prop:sufficient_for_assumption.Item4} $A$ is compact.
\end{enumerate}
\end{proposition}

\begin{proof}
We address each case individually.

\eqref{prop:sufficient_for_assumption.Item1}:  If $\varphi$ attains its norm, we can consider a constant sequence of functions equal to some norming function for $\varphi$.

\eqref{prop:sufficient_for_assumption.Item2}: If $M$ is proper, then $rM=M$ (Example~\ref{Example.srM.proper}). Therefore, as each of the functions  $\overline{h_n}_A $ and $ \underline{h_n}_A$ is  $1$-Lipschitz on $M$, Assumption~\ref{Assumption_Pavlos}\ref{assumption_ii}  comes for free for any sequence $(h_n)_{n=1}^{\infty}\sbs B_{\Lip_0(M)}$. It is then enough to choose a sequence $(h_n)_n$  satisfying Assumption~\ref{Assumption_Pavlos}\ref{assumption_i}.

\eqref{prop:sufficient_for_assumption.Item4}: Take any sequence $(h_n)_{n=1}^\infty\sbs B_{\Lip_0(M)}$ such that $\|\varphi\|=\lim_{n\to\infty}\langle h_n,\varphi\rangle$. If $A$ is compact, by the Arzel\`a--Ascoli theorem we may choose a subsequence, denoted $(h_n)_{n=1}^\infty$ again, and an $h\in B_{\Lip_0(M)}$ satisfying $$\lim_{n\to\infty}\sup_{x\in A}\left|h_n(x)-h(x)\right|=0.$$ It follows straightforwardly that $(\overline{h_n}_A)_n $ and  $( \underline{h_n}_A)_n$ converge uniformly on $rM$ to $\overline{h}_A$ and $\underline{h}_A$, respectively. So, given $\varepsilon>0$, there is $n_0\in\N$ such that for each $n\geq n_0$ we have 
\[\sup_{\zeta\in rM}\left|\overline{h_n}_A(\zeta)-\overline{h}_A(\zeta)\right|<\varepsilon/3\quad\textup{ and }\quad\sup_{\zeta\in rM}\left|\underline{h_n}_A(\zeta)-\underline{h}_A(\zeta)\right|<\varepsilon/3.\]
Then, if $\xi\in rM$, choose an open set $U\sbs uM$ containing $\xi$ such that \[\left|\overline{h}_A(\xi)-\overline{h}_A(x)\right|<\varepsilon/3\ \text{ and } \left|\underline{h}_A(\xi)-\underline{h}_A(x)\right|<\varepsilon/3\]  {and
\[\left|\overline{h_n}_A(\xi)-\overline{h_n}_A(x)\right|<\varepsilon/3\ \text{ and } \left|\underline{h_n}_A(\xi)-\underline{h_n}_A(x)\right|<\varepsilon/3\ \text{ for all }n\in\{1,\ldots, n_0\}\]}
whenever $x\in U\cap M$. Then, for every  {$n\in\N$} and every $x\in U\cap M$, we have
\begin{align*}
\left|\overline{h_n}_A(\xi)-\overline{h_n}_A(x)\right|&\leq\left|\overline{h_n}_A(\xi)-\overline{h}_A(\xi)\right|+\left|\overline{h}_A(\xi)-\overline{h}_A(x)\right|\\
&\quad\quad+\left|\overline{h}_A(x)-\overline{h_n}_A(x)\right|\\&<\varepsilon,     
\end{align*}
and analogously $\left|\underline{h_n}_A(\xi)-\underline{h_n}_A(x)\right|<\varepsilon$.    
\end{proof}

\section{\texorpdfstring{The De Leeuw representation of elements in $\F(M)^{**}$ and supports}{The De Leeuw representation of elements in and supports}}
Our goal in this section is to prove Theorem~\ref{prop:shadow_almost_norm_attaining}, which reveals a relation between the support of a functional in $\F(M)^{**}$ and the support of its De Leeuw representing measure. Recall from \S\ref{ss:results} that for all $\varphi\in\F(M)^{**}$, the optimal representation set $\mathcal{M}_{op}(\varphi) \sbs \mathcal{M}(\beta\widetilde M)$ is nonempty. It will be important for us to understand the support of measures in $\mathcal{M}_{op}(\varphi)$. The {\it support} of a positive $\mu \in \mathcal M(\beta \widetilde M)$ is given by 
\[\supp(\mu) \coloneqq \left\{x\in \beta \widetilde M\colon \mu(U)>0\ \text{ for all open subsets } U \text{ containing } x\right\}.\]

While the elements in $\mathcal{M}_{op}(\varphi)$ are measures in $\beta \widetilde M$, our Lipschitz-free spaces are built over $M$. So, we must understand how the supports of elements in $\mathcal{M}_{op}(\varphi)$ are perceived in  $M$. For that, for each $i\in \{1, 2\}$, let 
\[\mathfrak{p}_i:\beta\widetilde{M}\to uM \] be the   continuous extension of the $i$-th coordinate projection of $\widetilde{M}$ onto $M$. Let also   \[\mathfrak{p}:\beta\widetilde{M}\to uM \times uM\] be the continuous extension of the inclusion map of $\widetilde{M}$ into $M\times M$. So, \[\mathfrak p(\eta)=(\mathfrak p_1(\eta),\mathfrak p_2(\eta))\]
for all $\eta\in \beta \widetilde M$.

The following theorem can be viewed as an extension of {\cite[Lemma 2.3 and Proposition 5.6]{APS1}} to certain functionals in $\F(M)^{**}$ that are not necessarily weak$^*$ continuous.

\begin{theorem}
\label{prop:shadow_almost_norm_attaining}
Let $(M,d,0)$ be a complete pointed metric space, $A\sbs M$ be closed, and $\varphi\in\mathfrak{D}(A)$. If Assumption~\ref{Assumption_Pavlos} holds and $\varphi$ avoids infinity, then 
\[\mathfrak{p}\left(\supp (\mu)\right)\,\sbs\,\, \gex^u(A)\times \gex^u(A)\] for all  $\mu\in \mathcal M_{op}(\varphi)$.
Moreover, if $M$ is proper and $A$ is compact, then
\[\mathfrak{p}\left(\supp (\mu)\right)\,\sbs\,\, \gex(A)\times \gex(A).\]\end{theorem}

Although the stronger statement about proper metric spaces is not used for our main results, it yields Theorem \ref{thm:L_projection_supports}, and we include it for potential future applications. The proof of Theorem~\ref{prop:shadow_almost_norm_attaining} is technical and rather long. For clarity, we {start by establishing two auxiliary lemmas. The first one is analogous to a part of \cite[Lemma 2.3]{APS1} and it shows that any norming sequence of a functional $\varphi$ in $\F(M)^{**}$ must ``approach maximal slope'' at all points of the support of any representing measure from $\mathcal M_{op}(\varphi)$.

\begin{lemma}
\label{lem:strategy_prop}
Let $(M,d,0)$ be a complete pointed metric space, $\varphi\in\F(M)^{**}$, $\mu\in\mathcal M_{op}(\varphi)$ and $\eta\in\beta\widetilde{M}$. If there is a sequence $(f_n)_{n=1}^\infty\sbs B_{\Lip_0(M)}$ satisfying  $\|\varphi\|=\lim_{n\to\infty}\langle f_n,\varphi\rangle$ and for which some neighborhood $ W\sbs \beta\widetilde{M}$ of $\eta$ satisfies  
\begin{equation*}
    \label{eq:norming_sequence}
\sup\{\Phi f_n(\tau)\colon\,n\in\N, \tau\in W\}< 1,
\end{equation*}
then $\eta\notin\supp(\mu)$.
\end{lemma}}

\begin{proof}
Let $(f_n)_n$ be such a norming sequence and $W \sbs \beta\widetilde{M}$ a neighborhood of $\eta$ satisfying $\sup\{\Phi f_n(\tau)\colon\,n\in\N, \tau\in W\}<r$ for some $r\in(0,1)$. Then
\begin{align*}
\|\varphi\|&=\lim_{n\to\infty}\langle f_n,\varphi\rangle\\
&=\lim_{n\to\infty}\int_{\beta\widetilde{M}}\Phi f_n  d\mu\\
&\leq\mu(\beta\widetilde{M}\setminus W)+r\mu(W).
\end{align*}
Since  $\mu\in\mathcal{M}_{op}(\varphi)$, $\|\varphi\|=\|\mu\|= \mu(\beta\widetilde{M})$. Hence, we must have $\mu(W)=0$. As $W$ is a neighborhood of $\eta$, this implies that $\eta\notin\supp(\mu)$.\end{proof}

Hence, our strategy for the proof of Theorem \ref{prop:shadow_almost_norm_attaining} will be to show that, for a set of points $\eta \in \beta\widetilde{M}$ of full measure $\mu$, if at least one coordinate of $\eta$ does not belong to $\gex^u(A)$, then $\varphi$ admits a norming sequence satisfying the hypothesis of Lemma~\ref{lem:strategy_prop} for $\eta$. The main technicalities of the argument are isolated in the next lemma.

\begin{lemma}\label{Lemma.technical.prop}
    Let $(M,d,0)$ be a complete pointed metric space, $A\sbs M$ be closed with $0\in A$. Let $\eta\in \beta \widetilde M$ with \begin{align}\label{Eq.Lemma.technical.prop.1} \mathfrak{p}(\eta)\in (rM\times rM)\setminus  (\gex^u(A)\times  \gex^u(A)) \end{align} and let  $(h_n)_{n=1}^\infty$ be a sequence in  $ B_{\Lip_0(M)}$ such that $(\overline{h_n}_A)_n$ and $(\underline{h_n}_A)_n$ are both equicontinuous at both $\mathfrak p_1(\eta)$ and $\mathfrak p_2(\eta)$. Then, there are $(f_n)_{n=1}^\infty\sbs B_{\Lip_0(M)}$ and a neighborhood $W{\sbs \beta\widetilde{M}}$ of $\eta$ such that:
    \begin{enumerate}
        \item For all $n\in\N$, the functions $h_n$ and $f_n$ coincide on a neighborhood of $A$. In particular, $\langle h_n,\varphi\rangle=\langle f_n,\varphi\rangle$ for all $n\in\N$ and all $\varphi\in\mathfrak{S}(A)$.
        \item $\sup\{\Phi f_n(\tau)\colon n\in\N, \tau \in W\}<1$.
    \end{enumerate}

Assume furthermore that $M$ is proper and $A$ is compact. Then, the same conclusion holds with \eqref{Eq.Lemma.technical.prop.1} replaced by the weaker condition that \begin{align}\label{Eq.Lemma.technical.prop.2} \mathfrak{p}(\eta)\in (rM\times rM)\setminus  (\gex(A)\times  \gex(A)). \end{align}
\end{lemma}

\begin{proof}
The proof will consist of modifying the functions  $h_n$  outside some neighborhood of $A$ so that the resulting sequence satisfies the required properties.  

By the hypothesis on $\eta$,  $\mathfrak{p}(\eta)=(\mathfrak{p}_1(\eta),\mathfrak{p}_2(\eta))$ is not in $\gex^{u}(A)\times \gex^{u}(A)$. Hence, either 
$\mathfrak{p}_1(\eta)$ or $\mathfrak{p}_2(\eta)$ is not in $\gex^u(A)$. Assume that  $\mathfrak{p}_1(\eta)\notin \gex^{u}(A)$ -- the other case will follow similarly, and we highlight the differences at the end of this proof.

As $\mathfrak{p}_1(\eta)\notin \gex^{u}(A)$ and as 
\begin{equation}\label{Eq.10Aug26.1} 
\gex^u(A)=\{\xi\in uM \colon g_A (\xi)=0 \} 
\end{equation}
(see Lemma~\ref{lem:gap function vs approximate intervals}), we have $g_A(\mathfrak{p}_1(\eta))>0$. By the continuity of the McShane gap function $g_A$, there is an open $U\sbs uM$ containing $\mathfrak p_1(\eta)$ such that 
\[\delta \coloneqq \frac{1}{2}\inf_{\zeta\in U}g_A(\zeta)>0.\]
By the definition of  $g_A$,
\[\overline{h_n}_A(\zeta)-\underline{h_n}_A(\zeta)\geq \delta \]
for all $\zeta\in U$ {and all $n\in\N$}. Let $r\in (0,\delta/8)$ be small, to be chosen later depending only on $A$ and $\eta$. Then, letting \[D \coloneqq \overline{B\left(A,r\right)}\sbs M ,\]
\eqref{eq:extensions_from_neighbourhoods}  yields that 
\begin{align}\label{Eq.10Aug26.1.Lemma}
    \left|{\overline{h_n}}_{D}(x)-{\overline{h_n}}_{A}(x)\right|\leq 2r \ \text{ and }\ \left|{\underline{h_n}}_{A}(x)-{\underline{h_n}}_{D}(x)\right|\leq 2r,
\end{align}
for all $x\in M$ {and all $n\in\N$},
and we have 
\begin{align}
    \label{eq:gap_size}
{\overline{h_n}}_{D}(x)-{\underline{h_n}}_{D}(x)
\geq \delta-4r > \frac{\delta}{2} 
\end{align}
for all $x\in U\cap M$ and all $n\in\N$. In particular, $U\cap D=\emptyset$.

We split the next part of the proof into two cases.

\begin{Case1} $\mathfrak{p}_1(\eta)\neq \mathfrak{p}_2(\eta)$.
\end{Case1}

As  $\mathfrak{p}_1(\eta)$ and $ \mathfrak{p}_2(\eta)$ are distinct and are in  $rM$,  we can pick disjoint open subsets $O,Q\sbs uM$ containing  $\mathfrak{p}_1(\eta)$ and $\mathfrak{p}_2(\eta)$, respectively, and $R>0$   such that 
\[d(\xi,0)\leq R \ \text{ for all } \ \xi\in O\cup Q.\] We can assume furthermore that
\[a \coloneqq d(O\cap M, Q\cap M)>0\]
(see Proposition~\ref{Lemma.AP_measures.sM.Basics}).  Let 
\[\eps\coloneqq \frac{1}{2}\min\{\delta,a\}\ \text{ and }\ r\coloneqq \frac{\eps}{40}.\]

Using that   $(\overline{h_n}_A)_n$
 is equicontinuous at  $\mathfrak p_1(\eta)$ and  $\mathfrak p_2(\eta)$,  we can pick open  $G, H\sbs uM$ containing   $\mathfrak{p}_1(\eta)$ and  $\mathfrak{p}_2(\eta)$, respectively, such that
\begin{equation}
    \label{eq:equisemicon}
\left|{\overline{h_n}}_{D}(x)-{\overline{h_n}}_{D}(\mathfrak{p}_1(\eta))\right|<5r,
\end{equation}
for all $n\in\N$ and all  $ x\in G\cap M$, and 
\begin{equation}
    \label{eq:equisemicon.2}
\left|{\overline{h_n}}_{D}(x)-{\overline{h_n}}_{D}(\mathfrak{p}_2(\eta))\right|<5r
\end{equation}
for all $n\in\N$ and all  $ x\in H\cap M$. Indeed, notice that the aforementioned equicontinuity gives \eqref{eq:equisemicon} and \eqref{eq:equisemicon.2} with $A$ replacing $D$   and with a better   $r$ instead of $5r$  (or any other positive number). Then, using \eqref{Eq.10Aug26.1.Lemma}, we can obtain    \eqref{eq:equisemicon} and \eqref{eq:equisemicon.2}.

With the open sets defined above, we now set 
\[V_1 \coloneqq U\cap O\cap G,\  V_2 \coloneqq Q\cap H , \text{ and }\ W \coloneqq \mathfrak{p}^{-1}(V_1\times V_2).\]
So, $V_1$ and $V_2$ are neighborhoods of $\mathfrak p_1(\eta)$ and $\mathfrak p_2(\eta)$ in $uM$, respectively, and  $W$ is a neighborhood of $\eta$ in $\beta \widetilde M$.
We now move towards defining $(f_n)_n$. Fix $n\in\N$ and suppose first that 
\begin{equation}
\label{eq:difference_in_coordiantes}
{\overline{h_n}}_{D}(\mathfrak{p}_1(\eta))-{\overline{h_n}}_{D}(\mathfrak{p}_2(\eta))>\frac{3\eps}{4}.
\end{equation}
Define the extension $g_n:(D\cup  V_1 \cup V_2)\cap M \to\R$ of the function $h_n\vert_{D}$ by letting
\[
g_n(x) \coloneqq \begin{cases}
    h_n(x),&x\in D,\\
    {\overline{h_n}}_{D}(\mathfrak{p}_1(\eta))-\frac{\varepsilon}{2},&x\in V_1\cap M ,\\
    {\overline{h_n}}_{D}(x),&x\in V_2 \cap M.\\
    \end{cases}
    \]
Notice that $g_n$ is well defined. Indeed, this follows immediately since  $V_1$ and $V_2$ are disjoint, ${\overline{h_n}}_{D}$ coincides with $h_n$ on $D$,   and $V_1\cap D=\emptyset$. 

Let us notice that  $\Lip(g_n)\leq 1$. As both functions $g_n\vert_{V_1\cap M }$ and $g_n\vert_{D\cup( V_2\cap M)}$ are clearly $1$-Lipschitz on their domains, it is enough to estimate $|g_n(x)-g_n(y)|$ for $x\in V_1\cap M$ and $y\in D\cup( V_2\cap M)$. Fix such $x$ and $y$. If  $y\in D$, then  
\begin{align*}
\nonumber
|g_n(x)-g_n(y)|=\left|{\overline{h_n}}_{D}(\mathfrak{p}_1(\eta))-\frac{\varepsilon}{2}-h_n(y)\right|.\end{align*}
If the term in the right-hand side above is positive, as $5r=\eps/8$,   \eqref{eq:equisemicon} gives
\begin{align*}\left|{\overline{h_n}}_{D}(\mathfrak{p}_1(\eta))-\frac{\varepsilon}{2}-h_n(y)\right|&={\overline{h_n}}_{D}(\mathfrak{p}_1(\eta))-\frac{\varepsilon}{2}-h_n(y)\\
&\leq  
{\overline{h_n}}_{D}(x)-h_n(y)-\frac{3\varepsilon}{8}\end{align*}
and, if this term is negative,  \eqref{eq:gap_size} and \eqref{eq:equisemicon} give
\begin{align*}\left|{\overline{h_n}}_{D}(\mathfrak{p}_1(\eta))-\frac{\varepsilon}{2}-h_n(y)\right|&=h_n(y)-{\overline{h_n}}_{D}(\mathfrak{p}_1(\eta))+\frac{\varepsilon}{2}\\
&\leq   
h_n(y)-{\overline{h_n}}_{D}(x)+\frac{5\varepsilon}{8}\\
&\leq h_n(y)-{\underline{h_n}}_{D}(x)-\frac{3\varepsilon}{8} . \end{align*} 
Hence, using again that $y\in D$,
\begin{align*} 
|g_n(x)-g_n(y)|&\leq \max
\left\{ {\overline{h_n}}_{D}(x)-h_n(y)-\frac{3\varepsilon}{8},h_n(y)-{\underline{h_n}}_{D}(x)-\frac{3\varepsilon}{8}\right\}\\  
&\leq \max\left\{ \Lip({\overline{h_n}}_{D}), \Lip({\underline{h_n}}_{D})\right\} d(x,y)\\
&\leq d(x,y).
\end{align*}
On the other hand, if   $y\in V_2\cap M$, then  \eqref{eq:equisemicon.2} and \eqref{eq:difference_in_coordiantes} give 
\[{\overline{h_n}}_{D}(\mathfrak{p}_1(\eta))-\frac{\varepsilon}{2}- {\overline{h_n}}_{D}(y)>\frac{3\varepsilon}{4}-\frac{\varepsilon}{2}-\frac{\varepsilon}{8}>0.\] Therefore,    \eqref{eq:equisemicon} gives
\begin{align*}
|g_n(x)-g_n(y)|
&={\overline{h_n}}_{D}(\mathfrak{p}_1(\eta))-\frac{\varepsilon}{2}- {\overline{h_n}}_{D}(y)\\
&\leq {\overline{h_n}}_{D}(x)-{\overline{h_n}}_{D}(y)-\frac{3\varepsilon}{8}\\
&\leq\Lip({\overline{h_n}}_{D})d(x,y)\\
&\leq d(x,y).
\end{align*}

As $g_n$ is $1$-Lipschitz, it can be extended to a $1$-Lipschitz function, say $f_n:M\to\R$. Note that $f_n$ coincides with $h_n$ on the neighborhood $D$ of $A$, satisfying item (1) of the lemma.

If $(x,y)\in W\cap \widetilde{M}$, then $x=\mathfrak{p}_1(x,y)\in V_1\cap M $ and $y=\mathfrak{p}_2(x,y)\in V_2\cap M $. Therefore, using \eqref{eq:equisemicon} once again, it follows from our choice of $R$ that 
\begin{align*}
\Phi f_n(x,y)&=\frac{f_n(x)-f_n(y)}{d(x,y)}\\
&=\frac{g_n(x)-g_n(y)}{d(x,y)}\\
&=\frac{{\overline{h_n}}_{D}(\mathfrak{p}_1(\eta))-\varepsilon/2- {\overline{h_n}}_{D}(y)}{d(x,y)}\\
&\leq\frac{{\overline{h_n}}_{D}(x)-3\varepsilon/8- {\overline{h_n}}_{D}(y)}{d(x,y)}\\
&\leq 1-\frac{3\varepsilon}{16R}.
\end{align*}
So, $\Phi f_n\leq 1-3\varepsilon/16R$ on $W$. 

Now that we have dealt with the assumption in \eqref{eq:difference_in_coordiantes}, suppose the contrary, i.e., suppose $n\in\N$ is such that 
\begin{equation}\label{eq:difference_in_coordiantes.2}{\overline{h_n}}_{D}(\mathfrak{p}_1(\eta))-{\overline{h_n}}_{D}(\mathfrak{p}_2(\eta))\leq \frac{3\varepsilon}{4}.\end{equation} 
 In this case, we simply define  
\[f_n \coloneqq {\overline{h_n}}_{D} \ \text{ on the whole } M.\] This is clearly a $1$-Lipschitz function agreeing with $h_n$ on $D$.
Additionally, from \eqref{eq:equisemicon}, \eqref{eq:difference_in_coordiantes.2}, {and the definition of $a$,} it follows that for any $(x,y)\in W\cap \widetilde{M}$, we have
\begin{align*}
\Phi f_n(x,y)&=\frac{f_n(x)-f_n(y)}{d(x,y)}\\
&=\frac{{\overline{h_n}}_{D}(x)-{\overline{h_n}}_{D}(y)}{d(x,y)}\\
&\leq\frac{{\overline{h_n}}_{D}(\mathfrak{p}_1(\eta))+\varepsilon/8- {\overline{h_n}}_{D}(\mathfrak{p}_2(\eta))+\varepsilon/8}{d(x,y)}\\ &\leq \frac{\varepsilon}{a}.
\end{align*}
The inequality $\Phi f_n\leq \varepsilon/a$ thus holds on the entire $W$.

Summing up, by our choice of $\eps$, 
\[\sup\{\Phi f_n(\tau)\colon\,n\in\N, \tau\in W\}\leq \max\left\{1-\frac{3\varepsilon}{16R},\frac{\varepsilon}{a}\right\}<1.\]

\begin{Case2}
$\mathfrak{p}_1(\eta)=\mathfrak{p}_2(\eta)$.
\end{Case2}
 
 In this case, we simply take \[\eps\coloneqq\frac{\delta}{2}\ \text{ and }\ r\coloneqq\frac{\eps}{40}.\]
 As in   Case 1,  find an open set $G\sbs uM $ containing $\mathfrak{p}_1(\eta)$ such that 
\begin{equation}
\label{eq:equisemicon_case2}
\left|{\overline{h_n}}_{D}(x)-{\overline{h_n}}_{D}(\mathfrak{p}_1(\eta))\right|<5r,
\end{equation}
for all $n\in\N$ and all $x\in G\cap M$. We then set
\[V \coloneqq U\cap G .\]

For each $n\in\N$, define $g_n:D\cup (V\cap M)\to\R$ by 
\[g_n(x) \coloneqq \begin{cases}
    h_n(x),&x\in D,\\
    {\overline{h_n}}_{D}(\mathfrak{p}_1(\eta))-\frac{\varepsilon}{2},&x\in V\cap M.
    \end{cases}
    \]
The map $g_n$ is well defined since $D\cap V=\emptyset$ and, proceeding exactly as in Case 1, we have that, for all  $x\in V\cap M$ and all  $y\in D$,
\begin{align*}
|g_n(x)-g_n(y)|&=\left|{\overline{h_n}}_{D}(\mathfrak{p}_1(\eta))-\frac{\varepsilon}{2}-h_n(y)\right|\\
&\leq \max\left\{{\overline{h_n}}_{D}(x)-h_n(y)-\frac{3\varepsilon}{8},h_n(y)-{\underline{h_n}}_{D}(x)-\frac{3\varepsilon}{8}\right\}\\
&\leq \max\left\{\Lip({\overline{h_n}}_{D}),\Lip({\underline{h_n}}_{D})\right\}d(x,y)\\
&\leq d(x,y).
\end{align*}
So, $g_n$ is $1$-Lipschitz. 

Extend $g_n$ to a $1$-Lipschitz function, say $f_n:M\to\R$. Then again in this case, $f_n$ agrees with $h_n$ on the neighborhood $D$ of $A$. Moreover, if we put \[W \coloneqq \mathfrak{p}^{-1}(V\times V),\] then $W$ is a neighborhood of $\eta$ in $\beta\widetilde{M}$ satisfying
\[\Phi f_n(\tau)=0\ \text{ for all }\ \tau\in W\ \text{ and all }\  n\in \N.\]
This concludes the proof of Case 2.
\\

This finishes the proof of the first statement of the lemma in the case $\mathfrak{p}_1(\eta)\notin \gex^{u}(A)$. As we briefly mentioned above, the proof follows similarly if   $\mathfrak{p}_2(\eta)\notin \gex^{u}(A)$. The differences lie in the fact that the map $\Phi f_n$ is not symmetric, so, some minor changes are necessary. Indeed,  we should replace   ${\overline{h_n}}_{D}$ and ${\overline{h_n}}_{D}(\mathfrak{p}_1(\eta))-\varepsilon/2$ with ${\underline{h_n}}_{D}$ and  ${\underline{h_n}}_{D}(\mathfrak{p}_2(\eta))+\varepsilon/2$, respectively, wherever necessary.

For the last statement of the lemma, assume that $M$ is proper and that $A$ is compact. In this case, Lemma~\ref{lem:gap function vs approximate intervals} implies that
\begin{equation}\label{Eq.10Aug26.2} \gex(A)=\{\xi\in uM\colon g_A(\xi)=0\}.
\end{equation}
The argument then proceeds exactly as above with \eqref{Eq.10Aug26.2} replacing 
\eqref{Eq.10Aug26.1}.
\end{proof}

\begin{proof}[Proof of Theorem~\ref{prop:shadow_almost_norm_attaining}]
By Proposition \ref{how support sets interact with change of base point}, for any $a\in A$, $T_{0,a}^*\varphi\in \mathfrak{S}(A)$, by Remark \ref{change of base point remarks} \eqref{change of base point remarks ii}, $\mathcal{M}_{op}(\varphi) = \mathcal{M}_{op}(T_{0,a}^*\varphi)$, and, by Remark \ref{change of base point remarks} \eqref{change of base point remarks iii} $T_{0,a}^*\varphi$ avoid infinity. Therefore, by perhaps considering $T_a^*\varphi$ instead of $\varphi$, we may assume, without loss of generality, that $0\in A$ and $\varphi\in\mathfrak{S}(A)$. Fix $\mu\in \mathcal M_{op}(\varphi)$. Define  
\begin{align*}
\mathcal{R}& \coloneqq \mathfrak{p}^{-1}(rM\times rM).
\end{align*}
So, $\mathcal{R}$  is the open and dense subset of $\beta\widetilde{M}$ consisting of all points in $\beta\widetilde{M}$ that are limits of bounded nets in $\widetilde{M}$. Since $\varphi$ avoids infinity, \cite[Corollary 2.12]{APS2} implies that \[\|\mu\|=\mu(\mathcal{R}),\]  
$\mu$ has no mass on subsets outside $\mathcal R$.{ Moreover, because $\mu$ is inner regular, it has no mass outside $\supp(\mu)$ either.} This implies 
\[\supp(\mu)\sbs \overline{\mathcal{R}\cap\supp(\mu)}^{\beta\widetilde{M}}.\] 
As $\gex^u(A)$ is closed in $uM$ and $\mathfrak p$ is continuous, we only need to prove that 
\begin{equation}
\label{eq:sufficient_inclusion}
\mathfrak{p}\left(\mathcal{R}\cap\supp (\mu)\right)\,\sbs\,\, \gex^{u}(A)\times \gex^{u}(A).
\end{equation}

In order to prove  \eqref{eq:sufficient_inclusion},  it is enough to pick $\eta\in \mathcal{R}$  such that either  $\mathfrak{p}_1(\eta)$ or $ \mathfrak{p}_2(\eta)$ is not in $\gex^{u}(A)$ and   show that   $\eta\notin\supp(\mu)$.   Fix such $\eta\in \mathcal{R}$.

We are left to notice the existence of a sequence $(f_n)_n$ as in Lemma \ref{lem:strategy_prop}. But this is an immediate consequence of the first conclusion of Lemma~\ref{Lemma.technical.prop}. Indeed, as we are under Assumption~\ref{Assumption_Pavlos}, fix  a sequence of functions $(h_n)_{n=1}^\infty$ as therein. Since this sequence satisfies the hypothesis of Lemma~\ref{Lemma.technical.prop} for $\eta$, this completes the proof.

The last statement of the proposition follows completely analogously with the second conclusion of Lemma~\ref{Lemma.technical.prop} being used instead of its first.
\end{proof}

\section{\texorpdfstring{$L$-projections on $\F(M)^{**}$}{-projections on}} \label{SectionLProjHer}

In this section we apply Theorem~\ref{prop:shadow_almost_norm_attaining} to the context of $L$-projections. Our main result is Theorem \ref{thm:complete_compact}. Besides using it to show that $\F(X)$ is not $L$-embedded in $\F(X)^{**}$ for all Banach spaces $X$ with dimension at least 2 in \S\ref{SectionnonLembBanach2}, this theorem is also used to show that, under certain conditions, $L$-embeddability of $\F(M)$ in its bidual passes to $L$-embeddability of $\F(A)$ in its bidual, where $A\sbs M$ (see \S\ref{Subsection.Hereditary}). In order to obtain these results, we will analyze what $L$-projections do to $\mathfrak D(A)$ and, more precisely, when they take $\mathfrak D(A)$ into $\F(A)$.

{Corollary \ref{cor:L_projection_shadow} below} shows that  Theorem~\ref{prop:shadow_almost_norm_attaining} holds in a broader sense. Precisely, we show that the conclusion of this proposition holds not only for functionals $\varphi$ as therein but also for images of such functionals under $L$-projections. {It is a consequence of the following result on non-expansion of supports of De Leeuw representing measures of functionals under $L$-projections:

\begin{proposition}
\label{prop:support_preservation_deLeeuw}
Let $(M,d,0)$ be a complete pointed metric space and let $P$ be an $L$-projection on $\F(M)^{**}$. If $\varphi\in \F(M)^{**}$ and $\nu\in\mathcal{M}_{op}(P\varphi)$, then there exists $\mu\in\mathcal{M}_{op}(\varphi)$ such that $\supp(\nu)\sbs \supp(\mu)$.
\end{proposition}
\begin{proof}
Let $\nu\in\mathcal{M}_{op}(P\varphi)$ and $\lambda\in\mathcal{M}_{op}((\Id-P)\varphi)$ -- as already noted in Definition~\eqref{DefiMopVarphi}, such measures always exist (see \cite[Proposition 3]{Aliaga}). We observe that
\[\mu \coloneqq \nu+\lambda\in\mathcal{M}_{op}(\varphi).\]
Indeed, clearly $\mu\geq 0$ and 
\[\langle f,\varphi\rangle=\langle f,P\varphi\rangle+\langle f,(\Id-P)\varphi\rangle=\int_{\beta\widetilde{M}}\Phi f\, d\nu+\int_{\beta\widetilde{M}}\Phi f\, d\lambda=\int_{\beta\widetilde{M}}\Phi f\, d\mu.\] Moreover, as these measures are positive and $P$ is an $L$-projection, we also have
\[\|\mu\|=\|\nu\|+\|\lambda\|=\|P\varphi\|+\|(\Id-P)\varphi\|=\|\varphi\|.\]
From the positivity of the measures again, it follows that $\supp(\nu)\sbs \supp(\mu)$.
\end{proof}}

\begin{corollary}
\label{cor:L_projection_shadow}
Let $(M,d,0)$ be a complete pointed metric space, $A\sbs M$ be closed and $\varphi\in\mathfrak{D}(A)$. Assume that Assumption~\ref{Assumption_Pavlos} holds and that $\varphi$ avoids infinity. Then, for all $L$-projections $P$ on $\F(M)^{**}$ and all $\nu\in\mathcal{M}_{op}(P\varphi)$, we have \[\mathfrak{p}(\supp(\nu))\sbs \gex^{u}(A)\times\gex^{u}(A).\]
Moreover, if $M$ is proper and $A$ is compact, then for any $\nu\in\mathcal{M}_{op}(P\varphi)$, we obtain \[\mathfrak{p}(\supp(\nu))\sbs \gex(A)\times\gex(A).\]
\end{corollary}

\begin{proof}
Let $\nu\in\mathcal{M}_{op}(P\varphi)$ be arbitrary and let $\mu\in\mathcal{M}_{op}(\varphi)$ be given by Proposition \ref{prop:support_preservation_deLeeuw}. Then,
Theorem~\ref{prop:shadow_almost_norm_attaining} implies that 
\[
\mathfrak{p}(\supp (\nu))\sbs \mathfrak{p}(\supp (\mu))\sbs \gex^{u}(A)\times \gex^{u}(A)
\]
and, if moreover $M$ is proper and $A$ is compact, the same proposition gives
\[
\mathfrak{p}(\supp (\nu))\sbs \mathfrak{p}(\supp (\mu))\sbs \gex(A)\times \gex(A).
\]  
This finishes the proof.
\end{proof}

Corollary~\ref{cor:L_projection_shadow} allows us to show that, under suitable conditions on {$M$ and} $A\sbs M$, $L$-projections on $\F(M)^{**}$ leave $\mathfrak D(A)$ invariant. Geodesic closedness of $A$ will be essential for the next result. Later, we will show that if $P$ is an $L$-projection with image in $\F(M)$, then, under much weaker conditions on $A$, we can conclude furthermore that $P$ sends $\mathfrak D(A)$ into $\F(A)$ (see Theorem~\ref{thm:complete_compact}).
 
\begin{theorem}
\label{thm:L_projection_supports}
Let $(M,d,0)$ be a proper pointed metric space and $A\sbs M$ be compact and geodesically closed. Then, for any $L$-projection $P\colon \F(M)^{**}\to \F(M)^{**}$, we have  
\[P\left(\mathfrak{D}(A)\right)\sbs \mathfrak{D}(A).\]
\end{theorem}

\begin{proof}
Let $\varphi\in\mathfrak{D}(A)$. As $A$ is compact, Assumption \ref{Assumption_Pavlos} holds by Proposition~\ref{prop:sufficient_for_assumption} and $\varphi$ avoids infinity. Since $M$ is proper and $\gex(A)=A$, Corollary~\ref{cor:L_projection_shadow} gives
that 
\[
\mathfrak{p}(\supp (\nu))\sbs \gex(A)\times \gex(A)=A\times A
\]
for all $\nu\in\mathcal{M}_{op}(P\varphi)$.
Therefore, if $f\in\Lip_0(M)$ is constant on some neighborhood of $A$, then $\Phi f=0$ on $\mathfrak{p}^{-1}(A\times A)$ and, for $\nu \in \mathcal{M}_{op}(P\varphi)$, \[\langle f,P\varphi\rangle=\int_{\mathfrak{p}^{-1}(A\times A)}\Phi f\,d\nu=0.\] Thus, $P\varphi\in\mathfrak{D}(A)$.
\end{proof}

We shall now drop the conditions { of properness of $M$ and} geodesic closedness of $A\sbs M$ and explore further consequences of  Corollary~\ref{cor:L_projection_shadow}. We still assume $A$ to be compact so that Assumption~\ref{Assumption_Pavlos} holds more broadly (see Proposition~\ref{prop:sufficient_for_assumption}) and  Corollary~\ref{cor:L_projection_shadow} applies. We also only consider $L$-projections with image in $\F(M)$, which yields $\gex(A)$ instead of $\gex^u(A)$ in the conclusion.

\begin{lemma}
\label{lem:proper_compact_gex}
Let $(M,d,0)$ be a complete pointed metric space, $A\sbs M$ a non-empty compact subset and let $\varphi\in\mathfrak{D}(A)$. If $P:\F(M)^{**}\to\F(M)$ is an $L$-projection, then $P\varphi\in\F(\gex(A)\cup\{0\})$. 
\end{lemma}

The proof of Lemma \ref{lem:proper_compact_gex}  makes use of a lemma by  Aliaga. Since this lemma will be used repeatedly below, we provide its statement for the reader's convenience.

\begin{lemma}\emph{(}\cite[Lemma 8]{Aliaga}\emph{).}
Let $(M,d,0)$ be a complete pointed metric space, $m \in \F(M)$, and  $\mu\in\mathcal M(\beta\widetilde  M)$ be such that $\Phi^*\mu  = m$. Then \[\supp(m) \sbs \mathfrak p_1(\supp(\mu))\cup \mathfrak p_2(\supp(\mu)).\]\label{Lemma.Aliaga.8}
\end{lemma}

\begin{proof}[Proof of Lemma \ref{lem:proper_compact_gex}] 
By the compactness of $A$, the functional $\varphi$ must avoid infinity, and we may therefore apply Corollary~\ref{cor:L_projection_shadow} to obtain that
\begin{equation}\label{Eq.11Aug26.1tilde}\mathfrak{p}(\supp(\nu))\sbs\gex^{u}(A)\times\gex^{u}(A)\ \text{ for all }\   \nu\in\mathcal{M}_{op}(\varphi).
\end{equation}
Lemma \ref{Lemma.Aliaga.8} gives  that 
\[\supp (P\varphi) \sbs (\mathfrak{p}_1(\supp( \nu))\cup\mathfrak{p}_2(\supp (\nu)))\cap M.\] By \eqref{Eq.11Aug26.1tilde}, we conclude that 
\begin{align*}
\supp (P\varphi) \sbs  \gex^{u}(A)\cap M
=\gex(A)
\end{align*}
and hence  \begin{align}\label{Eq.11.Aug.26.4tilde}
P\varphi\in\F(\gex(A)\cup\{0\}).
\end{align}
\end{proof}

We now prove one of our main results about support preservation. The next theorem shows that Lemma~\ref{lem:proper_compact_gex} can be strengthened so that the resulting element $P\varphi$ lies in $\F(A)$ and not only in $\F(\gex(A)\cup\{0\})$. The basic idea behind the improvement is to use the compactness of $A$ for decomposing functionals into functionals with supports so small that their geodesic extensions lie in an arbitrarily small neighborhood of $A$. Notice that, for $L$-projections with image in $\F(M)$, this theorem also strengthens Theorem~\ref{thm:L_projection_supports} above.

\begin{theorem}
\label{thm:complete_compact}
Let $(M,d,0)$ be a complete pointed metric space and $A\sbs M$ be a non-empty compact subset. If $P:\F(M)^{**}\to\F(M)$ is an $L$-projection, then \[P(\mathfrak S(A))\sbs\F(A\cup \{0\})\text{ and }P(\mathfrak D(A))\sbs\F(A).\]   
\end{theorem}

In the proof of Theorem~\ref{thm:complete_compact}, we will use the following observation.

\begin{proposition}
\label{prop:intersection}
Let  $(M,d,0)$ be a complete pointed metric space, $A\sbs M$ be a compact subset, and $B\sbs M$ be closed. Then, 
\[\mathfrak{S}(A)\cap\mathfrak{S}(B)=\mathfrak{S}(A\cap B).\]
\end{proposition}

\begin{proof}
As the inclusion ``$\sps$'' is immediate, we only prove the inclusion ``$\sbs$''. Let $\varphi\in \mathfrak{S}(A)\cap\mathfrak{S}(B)$. Since $A$ is a bounded set, $\varphi$ avoids infinity, and by \cite[Theorem 3.13]{AP_measures} there exists a smallest (in the sense of inclusion) closed subset $\mathcal{S}(\varphi)$ of $uM$ with the property that $\langle f,\varphi\rangle=0$ for any $f\in\Lip_0(M)$ with $f\vert_{U\cap M}=0$ for some neighborhood $U\sbs uM$ of $\mathcal{S}(\varphi)$. Namely, it is the set 
\[ \mathcal{S}(\varphi)=\bigcap\left\{\overline{C}^{uM}\colon C\sbs M \ \text{ and }\ \langle f,\varphi\rangle= 0\text{ for all }\ f\in \Lip_0(M)\text{ with }f\vert_C=0\right\}\]
as previously defined in \S\ref{ss:results}. The formula of $\mathcal{S}(\varphi)$ together with the fact that $A$ is compact and $B$ closed in $M$ implies that 
\[\mathcal{S}(\varphi)\sbs A\cap \overline{B}^{uM}=A\cap B.\]
In particular, by the property of $\mathcal{S}(\varphi)$ described above, $\varphi\in\mathfrak{S}(A\cap B)$. 
\end{proof}

\begin{proof}[Proof of Theorem~\ref{thm:complete_compact}]
Let $\varphi\in \mathfrak S(A)$ and fix an auxiliary $\eps\in (0,1)$. {For better readability, we split the proof into four steps.}
  
{\begin{Step1} Decomposition of $\varphi$. \end{Step1}}
Since $A$ is compact, we can find an $n\in\N$ and points $(x_i)_{i=1}^n\sbs A$ such that $A\sbs \bigcup_{i=1}^nB(x_i,\eps)$; write $B \coloneqq \bigcup_{i=1}^nB(x_i,\eps)$. For each $i\in\{1,\ldots,n\}$, define $g_i:M\to[0,1]$ by letting 
\[g_i(x) \coloneqq \min\left\{1,d(x,M\setminus B(x_i,2\eps))\right\}\ \text{ for all }\   x\in M.\] 
Then each $g_i$ is Lipschitz and  $\supp(g_i)\sbs B_i$, where $B_i \coloneqq \overline{B(x_i,2\eps)} $.  Moreover,    for every $x\in B$, there exists a $j\in\{1,\ldots,n\}$ such that $g_j(x)\geq\eps$. For each $i\in\{1,\ldots,n\}$, let
\[
h_i(x) \coloneqq \left\{\begin{array}{ll}
   \frac{g_i(x)}{\sum_{k=1}^ng_k(x)},&x\in B,\\
   0,&x\in M\setminus B(x_i,2\eps).
\end{array}\right.
\]
Notice that, since  $g_i=0$ on $M\setminus B(x_i,2\eps)$, the map $h_i$ is well defined. 

Let us observe that each $h_i$ is   Lipschitz. Clearly, $h_i\vert_{M\setminus B(x_i,2\eps)}$ is Lipschitz. Moreover, since each $g_k$ is Lipschitz and 
\begin{equation}\label{Eq.prop:complete_compact.1}\eps\leq\sum_{k=1}^ng_k(x)\leq n
\ \text{ for all }\ x\in B,\end{equation} each  $h_i\vert_B$ is also Lipschitz. We are left to estimate $|h_i(x)-h_i(y)|$ for  $x\in B$ and $y\in M\setminus B(x_i,2\eps)$. For such $x$ and $y$, using \eqref{Eq.prop:complete_compact.1},
\begin{align*}
|h_i(x)-h_i(y)|&=h_i(x)-h_i(y)\\
&=h_i(x)-g_i(y)\\
&\leq \frac{g_i(x)}{\eps}-g_i(y)
\\
&=\frac{g_i(x)-g_i(y)}{\eps}\\
&\leq \frac{\Lip(g_i)}{\eps}d(x,y).
\end{align*}

Since each  $h_i$ is Lipschitz, we can extend it to a Lipschitz function on the whole $M$ with support contained in $B_i$; by abuse of notation, we still denote these extensions by $h_i$. As $\sum_{i=1}^nh_i=1$ on $B$, for any function $f\in\Lip_0(M)$, we have that $f-f\sum_{i=1}^nh_i=0$ on $B$. Therefore, since $B$ is a neighborhood of $A$ and $\varphi\in\mathfrak{S}(A)$,
\begin{equation}\label{Eq.Thm.14.Aug.26.1}\varphi=\sum_{i=1}^nW_{h_i}^*(\varphi),
\end{equation}
where $W_{h_i}$ are given by \eqref{Eq.Wh.formula} above.   

{\begin{Step2} $W_{h_i}^*(\varphi)\in\mathfrak{S}(B_i\cap A)$ for all $1\leq i\leq n$. \end{Step2}}
Fix  $i\in\{1,\ldots,n\}$. {By Proposition~\ref{prop:intersection}, in order to prove that $W_{h_i}^*(\varphi)\in\mathfrak{S}(B_i\cap A)$}, it is enough to show that $W_{h_i}^*(\varphi)$ belongs to both $\mathfrak{S}(B_i)$ and $\mathfrak S(A)$. For the former, suppose  $f\in\Lip_0(M)$ vanishes on a neighborhood of $ B_i$. Then $h_if=0$ on $M$ and 
\[\langle f,W^*_{h_i}\varphi\rangle=\langle h_if,\varphi\rangle=0.\] 
So,  $W_{h_i}^*(\varphi)\in\mathfrak{S}(B_i)$. If, on the other hand, $f\in\Lip_0(M)$ vanishes on a neighborhood of $A$, then so does $h_if$. Therefore, as   $\varphi\in\mathfrak{S}( A)$, we must have  $W_{h_i}^*(\varphi)\in\mathfrak{S}( A)$.

{\begin{Step3} $P(W_{h_i}^*(\varphi))\in\F(\gex(B_i\cap A)\cup\{0\})$ for all $1\leq i\leq n$. \end{Step3}}
Fix $i\in\{1,\ldots,n\}$. By Proposition \ref{prop:decomposition_S(A)}, there exist $\lambda_i\in\mathbb{R}$, $x_i\in B_i\cap A$, and $\psi_i\in\mathfrak{D}(B_i\cap A)$, such that
\[W_{h_i}^*(\varphi) = \lambda\delta(x_i)+\psi_i\]
(if $0\in B_i\cap A$ then $\mathfrak{D}(B_i\cap A) = \mathfrak{S}(B_i\cap A)$, and so we can just let $\lambda_i=0$, $x_i = 0$, and $\psi_i = P(W_{h_i}^*(\varphi))$ in this case). By Lemma \ref{lem:proper_compact_gex}, $P\psi_i\in\mathcal{F}(\gex(B_i\cap A)\cup\{0\})$, and, because $P\delta(x_i) = \delta(x_i)\in\mathcal{F}(\gex(B_i\cap A)\cup\{0\})$, $PW_{h_i}^*(\varphi)\in\mathcal{F}(\gex(B_i\cap A)\cup\{0\})$.

{\begin{Step4} $ P\varphi\in  \F(A\cup\{0\})$. \end{Step4}}
Since  $\diam(B_i\cap A)\leq 4\eps$ for each $i\in\{1,\ldots,n\}$, we have 
\[\gex(B_i\cap A)\sbs\overline{\gex_{\eps}(B_i\cap A)}\sbs B(B_i\cap A,6\eps)\sbs B(A,6\eps).\]
Therefore, {Step 3} gives that 
\[P(W_{h_i}^*(\varphi))\in\F\left(\overline{B(A,6\eps)\cup\{0\}}\right).\]
By \eqref{Eq.Thm.14.Aug.26.1}, we conclude that 
\[P\varphi=\sum_{i=1}^nP(W_{h_i}^*(\varphi))\in\F\left(\overline{B(A,6\eps)\cup\{0\}}\right).\]
 Since $\eps>0$ was arbitrary, Theorem \ref{Theorem.APPP.2.1}  finally implies that $P\varphi\in  \F(A\cup\{0\})$ 
as desired.

Now let $\varphi\in\mathfrak{D}(A)$ and fix $a\in A$. By Proposition \ref{change of base point proposition} \eqref{change of base point proposition i} and \eqref{change of base point proposition ii}, $T^*_{0,a}:\F(M)\to\F_a(M)$ is a surjective linear isometry mapping $\F(M)$ onto $\F_a(M)$. It furthermore maps $\F(A) = \overline{\mathrm{span}}\{\delta(x)-\delta(y):x,y\in A\}\sbs\F(M)$ onto $\F_a(A) = \overline{\mathrm{span}}\{\delta(x)-\delta(y):x,y\in \gex(A)\}\sbs\F_a(M)$. Recall also that, by Proposition \ref{how support sets interact with change of base point}, $T_{0,a}^*\varphi\in\mathfrak{D}_a(A)$. We may, therefore, without loss of generality, assume that $0\in A$, and therefore, by the first part of this proof, $P\varphi\in\mathcal{F}(A\cup\{0\}) = \mathcal{F}(A)$.
\end{proof}

\subsection{Hereditariness of $L$-embeddability for Lipschitz-free spaces}\label{Subsection.Hereditary}
In this subsection, we present two cases of $A\sbs M$ in which the $L$-embeddability of $\F(M)$ passes to $\F(A)$. Namely, this will be the case if $A$ is compact (see Theorem~\ref{Thm.LembPassesAcompact}) {or if $A$ is geodesically closed (see Theorem~\ref{thm:L_embed_hereditary_geodesic})}.

\begin{theorem}\label{Thm.LembPassesAcompact}
Let $(M,d,0)$ be a complete pointed metric space and $A\sbs M$ be a compact subset. If $P:\mathcal{F}(M)^{**}\to \mathcal{F}(M)$ is an $L$-projection, then $P(\mathcal{F}(A)^{**}) \sbs \mathcal{F}(A)$. In particular, if $\mathcal{F}(M)$ is $L$-embedded in its bidual, then $\mathcal{F}(A)$ is $L$-embedded in its bidual.     
\end{theorem}

\begin{proof}
This is an immediate corollary of  Theorem~\ref{thm:complete_compact}. Indeed, this theorem says that $P(\mathfrak D(A))\sbs \F(A)$ and, as noticed in Remark \ref{subset free space bidual containement in constant vanishers}, $\F(A)^{**}\sbs \mathfrak{D}(A)$. Hence, $P\vert_{\F(A)^{**}}$ takes $\F(A)^{**}$ into $\F(A)$ and, being the restriction of an $L$-projection, it is also an $L$-projection. 
\end{proof}

The following lemma will be used in the proof of Theorem~\ref{thm:L_embed_hereditary_geodesic} to show that an $L$-projection of $\F(M)^{**}$ into functionals avoiding infinity necessarily annihilates functionals concentrated at infinity.

{
\begin{lemma}
\label{lem:infinity_norming}
Let $(M,d,0)$ be an unbounded complete pointed metric space and let $\varphi\in\F(M)^{**}$ be concentrated at infinity. Then, for every $r>0$ and $\varepsilon>0$, there exists an $h\in B_{\Lip_0(M)}$ such that $\supp(h)\sbs M\setminus B(0,r)$ and $\langle h,\varphi\rangle\geq \|\varphi\|-\varepsilon$.   
\end{lemma}

\begin{proof}
First, we observe that if $L>0$ is arbitrary and if $f,g\in\Lip_0(M)$ are such that $f\vert_{M\setminus {B(0,L)}}=g\vert_{M\setminus{B(0,L)}}$, then $\langle f,\varphi\rangle=\langle g,\varphi\rangle$. Indeed, choose $n\in\N$ so large that $\daleth_n\vert_{\overline{B(0,L)}}=1$. Then, since $\supp(f-g)\sbs \overline{B(0,L)}$, the fact that $\varphi$ is concentrated at infinity gives 
\[\langle f-g,\varphi\rangle=\langle (f-g)\daleth_n,\varphi\rangle=\langle f-g ,W_{\daleth_n}^*(\varphi)\rangle=0.\]

Fix  now $r>0$ and $\varepsilon>0$, and let $\delta \coloneqq \varepsilon/(1+\|\varphi\|)$. Pick an $R>r$ large enough so that 
\begin{equation}
\label{eq:choice_of_R}
\frac{t}{t-r}<1+\delta\ \text{ for every }\ t\geq R,
\end{equation}
and choose an $f\in B_{\Lip_0(M)}$ with $\langle f,\varphi\rangle\geq \|\varphi\|-\delta$.

Define a function $g:B(0,r)\cup (M\setminus B(0,R))\to\R$ by letting
\[ 
g(x) \coloneqq \begin{cases}
 0,& x\in B(0,r),\\
f(x),&x\in M\setminus B(0,R) .
\end{cases}
\]
Clearly, both $g\vert_{B(0,r)}$ and $g\vert_{M\setminus B(0,R)}$ are $1$-Lipschitz functions and if $x\in B(0,r)$ and $y\in M\setminus B(0,R)$, then by \eqref{eq:choice_of_R},
$$\frac{|g(x)-g(y)|}{d(x,y)}=\frac{|f(y)|}{d(x,y)}\leq\frac{d(y,0)}{d(x,y)}\leq \frac{d(y,0)}{d(y,0)-d(x,0)}\leq \frac{d(y,0)}{d(y,0)-r}<1+\delta.$$
Function $g$ can thus be extended to a $(1+\delta)$-Lipschitz function on $M$, which we will denote $g$ again. So, $g\in (1+\delta)B_{\Lip_0(M)}$ and, since $g\vert_{M\setminus{B(0,R)}}=f\vert_{M\setminus{B(0,R)}}$, the discussion in the beginning of the proof implies that $\langle g,\varphi\rangle=\langle f,\varphi\rangle$. Finally, let  $h \coloneqq (1+\delta)^{-1} g$. Then $h\in   B_{\Lip_0(M)}$,   \[\supp(h)\sbs\supp (g)\sbs M\setminus B(0,r)\] and, by our choice of $\delta$,  
\[\langle h,\varphi\rangle=\frac{1}{1+\delta}\langle g,\varphi\rangle\geq \frac{1}{1+\delta}(\|\varphi\|-\delta)\geq \|\varphi\|-\varepsilon.\]
 This completes the proof.
\end{proof}

We can now present the second case when $L$-embeddability is hereditary.
\begin{theorem}
\label{thm:L_embed_hereditary_geodesic}
Let $(M,d,0)$ be a complete pointed metric space and $A\sbs M$ be a geodesically closed subset. If $P:\mathcal{F}(M)^{**}\to \mathcal{F}(M)$ is an $L$-projection, then $P(\mathcal{F}(A)^{**}) \sbs \mathcal{F}(A)$. In particular, if $\mathcal{F}(M)$ is $L$-embedded in its bidual, then $\mathcal{F}(A)$ is $L$-embedded in its bidual. 
\end{theorem}

\begin{proof}
Fix $a\in A$. Because the operator $T_{0,a}^*:\F(M)^{**}\to\F_a(M)^{**}$ is an onto linear isometry preserving $\F(M)$, $\F(A)^{**}$, and $\F(A)$, we may assume, without loss of generality, that $0 = a\in A$. In particular, $\mathfrak{D}(A) = \mathfrak{S}(A)$. We begin the proof by showing that $P\varphi=0$ whenever $\varphi\in\F(M)^{**}$ is concentrated at infinity. If $M$ is bounded, zero is the only element in  $\F(M)^{**}$ which is concentrated at infinity. Suppose then that $M$ is unbounded and let $\varphi\in\F(M)^{**}$ be concentrated at infinity. Fix an auxiliary  $\varepsilon>0$. As $P$ has image in $\F(M)$, $P\varphi$ avoids infinity. So, there exists an $n\in\N$ such that 
\begin{equation}
\label{eq:avoids_infinity}
\|P\varphi-W_{\daleth_n}^*(P\varphi)\|<\varepsilon/2.
\end{equation} 
Lemma~\ref{lem:infinity_norming} gives an  $h\in B_{\Lip_0(M)}$ such that $\supp(h)\sbs M\setminus B(0,2^{n+1})$ and   $\langle h,\varphi\rangle>\|\varphi\|-\varepsilon/2$. Therefore, using that $P$ is an $L$-projection and recalling that $\daleth_n$ vanishes outside $B(0,2^{n+1})$, we have by \eqref{eq:avoids_infinity},
\begin{align*}
\|P\varphi\|&=\|\varphi\|-\|(\Id-P)\varphi\|\\
&< \|\varphi\|-\|\varphi-W_{\daleth_n}^*(P\varphi)\|+\varepsilon/2\\
&<\langle h,\varphi\rangle +\varepsilon/2-\langle h,\varphi-W_{\daleth_n}^*(P\varphi)\rangle+\varepsilon/2\\
&=\langle h\daleth_n,P\varphi\rangle+\varepsilon\\
&=\langle 0, P\varphi\rangle +\varepsilon\\
&=\varepsilon.
\end{align*}
Since $\varepsilon>0$ was arbitrary, we get that $P\varphi=0$.

We move now to general elements of $\F(A)^{**}$. Notice that, by Bishop--Phelps theorem, the norm-attaining functionals in $\F(A)^{**}$ are dense in this space. Therefore, in order to show that $P(\F(A)^{**})\sbs \F(A)$, it is enough to show that  $P\varphi\in \F(A)$ for all norm-attaining functionals $\varphi\in \F(A)^{**}$. Fix such $\varphi$. By Proposition~\ref{Prop.AP_measures.Cor.3.4}, we can write $\varphi=\varphi_M+\varphi_\infty$, where 
\[\varphi_M=\lim_n W^*_{\daleth_n}(\varphi)\ \text{ and } \ \varphi_\infty=\varphi-\varphi_M.\]
So,  $\varphi_M$ avoids infinity and $\varphi_\infty$ is concentrated at infinity, and  
\begin{equation}
\label{eq:functional_decomposition}
\|\varphi\|=\|\varphi_M\|+\|\varphi_\infty\|.
\end{equation}
By the previous paragraph, $P\varphi_\infty=0$, and so we only need to show that  $P\varphi_M \in \F(A)$. 

If $f\in B_{\Lip_0(M)}$   is norming for $\varphi$, then \eqref{eq:functional_decomposition} gives
\[\|\varphi\|=\langle f, \varphi\rangle=\langle f, \varphi_M\rangle + \langle f, \varphi_\infty\rangle\leq \|\varphi_M\|+\|\varphi_\infty\|=\|\varphi\|.\] Since this forces $\|\varphi_M\|=\langle f,\varphi_M\rangle$,  $\varphi_M$ also attains its norm. Notice also that $\varphi_M$ is in $\mathfrak S(A)$. Indeed, if $f$ vanishes on a neighborhood of $A$, then so does $f\daleth_n$ for all $n$. Therefore, as $\varphi\in \mathfrak S(A)$ (see Remark \ref{subset free space bidual containement in constant vanishers}), the definition of $\varphi_M$ gives that 
\[\langle f,\varphi_M\rangle=\lim_n\langle f,W_{\daleth_n}^*(\varphi)\rangle =\lim_n\langle f\daleth_n,\varphi\rangle=0.\]
By Proposition~\ref{prop:sufficient_for_assumption},   Assumption~\ref{Assumption_Pavlos} is satisfied for $\varphi_M$ and we can apply 
Corollary~\ref{cor:L_projection_shadow} to obtain 
\[\mathfrak{p}(\supp(\nu))\sbs \gex^{u}(A)\times\gex^{u}(A)\ \text{ for all }\   \nu\in\mathcal{M}_{op}(P\varphi_M).\]
Fix an arbitrary $\nu\in\mathcal{M}_{op}(P\varphi_M)$. As $P(\varphi_M)$ is in $\F(M)$, it follows from Lemma \ref{Lemma.Aliaga.8}  that
\begin{align*}
\supp (P\varphi_M) &\sbs (\mathfrak{p}_1(\supp( \nu))\cup\mathfrak{p}_2(\supp (\nu)))\cap M\\
&\sbs \gex^{u}(A)\cap M\\
&=\gex(A).    
\end{align*}
Since    $\gex(A)=A$, Proposition \ref{prop.APPP.2.6} gives that  $P\varphi_M\in \F(A\cup\{0\}) = \F(A)$.
\end{proof}
}

\section{\texorpdfstring{$\F(X)$ is not $L$-embedded in $\F(X)^{**}$ if $\dim(X)\geq 2$}{is not -embedded in  if }} \label{SectionnonLembBanach2}
In this penultimate section, we show how the arguments above can be used in order to rule out $L$-embeddability of the Lipschitz-free spaces of Banach spaces with dimension at least $2$. Besides the tools developed in the previous sections, a version of Godefroy--Kalton's lifting theorem  (Lemma~\ref{lem:2Dembedding})  as well as a condition implying that a pair of vectors is isometrically equivalent to two vectors in the canonical basis of $\ell_1$ will be needed (see Lemma~\ref{lem:l1basis}).

 The following lemma can be seen as a version to our setting of Godefroy--Kalton's lifting theorem (see \cite[Theorem 3.1 and Proposition 4.7]{GK}).

\begin{lemma} \label{lem:2Dembedding}
Let $X$ be a finite-dimensional normed space. Then for every $x_0 \in S_X$, there exists a linear map $T: X \to \F(X)^{**}$ such that
\begin{enumerate}[label=(\roman*)]
    \item\label{lem:2Dembedding.item1} $T$ is an isometric embedding,
    \item\label{lem:2Dembedding.item2} $T(y) \in {\mathfrak{S}([0,1]x_0)}$ for all $y \in X$,
    \item\label{lem:2Dembedding.item3} $T(x_0) \in \F(X)$, and
    \item\label{lem:2Dembedding.item4} $\|T(y)+{m}\| \geq \dist(y,\R x_0)$ for all $y \in X$ and ${m} \in \F([0,1]x_0)$.
\end{enumerate}
\end{lemma}

\begin{proof}
Let $x_0 \in S_X$. Before jumping into the details of the proof, let us give a brief overview of it. The idea is to map each $y \in X$ to an $X$-valued measure $\nu$ supported on $[0,1]x_0$ whose Radon--Nikodym derivative with respect to 1-dimensional Lebesgue measure is the constant vector $y$. Such a measure $\nu$ then acts on a Lipschitz function $f \in\Lip_0(X)$ via \[\langle f,\nu \rangle = \int_0^1 \partial_y f(sx_0) ds.\] The problem with making such a definition is that $\partial_y f(sx_0)$ need not be well defined, not even for almost every $s\in[0,1]$, and so a bit of care is required in making this approach work.  

We now start with the details. Let $\lambda$ denote a Lebesgue measure on $X$ and $(\phi_n: X \to \R)_{n\in\N}$ be a smooth, compactly supported approximate identity, i.e.,  each $\phi_n$ is in the class $C^\infty$, $\phi_n \geq 0$, $\int_{X} \phi_n d\lambda= 1$, and $\supp \phi_n \sbs B(0,1/n)$. For each $n\in\N$, define $T_n: X \to \F(X)^{**}$ by letting
\[\langle  f,T_n(y)  \rangle  \coloneqq  \int_0^1 \partial_y (\phi_n * f)(sx_0) ds ,\ \text{ for all } y\in X\ \text{ and all}\ f\in \Lip_0(X),\]
where in the formula above we are using the convolution 
\[(\phi_n * f)(z) \coloneqq \int_X \phi_n(z-x)f(x)d\lambda(x),\] and for a $C^\infty$ function $g: X \to \R$, the directional derivative $\partial_y g: X \to \R$ is given by \[\partial_y g(x)  \coloneqq  \lim_{h\to 0} \frac{g(x+hy)-g(x)}{h}.\]
It is straightforward that $T_n$ is a linear operator with norm at most $1$. Let now $\U$ be any nonprincipal ultrafilter on $\N$, and define $T: X \to \F(X)^{**}$ by
\[T(y)  \coloneqq  w^*\text{-}\lim_{n,\U} T_n(y)\ \text{ for all } y\in X.\]
Clearly, $T$ is an operator with $\|T\| \leq 1$. We now show items \ref{lem:2Dembedding.item1}--\ref{lem:2Dembedding.item4} hold for $T$.

  \ref{lem:2Dembedding.item1}: Let $y \in X$ and pick  $\psi \in S_{X^*}$ with $\psi(y) = \|y\|$, which is possible due to Hahn--Banach. Then,
\begin{align*}
    \|T(y)\| &\geq \langle \psi, T(y)  \rangle =  \lim_{n,\U}\int_0^1 \partial_y (\phi_n * \psi)(sx_0) ds = \int_0^1 \psi(y) ds = \|y\|.
\end{align*}
So, $T$ is an isometric embedding.

 \ref{lem:2Dembedding.item2}: The fact that $T(y) \in {\mathfrak{S}([0,1]x_0)}$, for all $y \in X$, follows from a  straightforward computation since $\supp \phi_n \sbs B(0,1/n)$ {and $[0,1]x_0$ is compact}.
 
  \ref{lem:2Dembedding.item3}:
Using the Fundamental Theorem of Calculus, we have
\[\langle f,T_n(x_0) \rangle = (\phi_n*f)(x_0)-(\phi_n*f)(0)\ \text{ for all }\  f \in \Lip_0(X)\ \text{ and all }\  n \in \N.\] But $ \phi_n*f$ converge pointwise to $f$ whenever $f$ is Lipschitz, so 
\[\langle f,T(x_0) \rangle = f(x_0)\ \text{ for every }\  f \in \Lip_0(X).\] Hence, $T(x_0) = \delta({x_0}) \in \F(X)$.

  \ref{lem:2Dembedding.item4}: Let $y\in X$ and, using Hahn--Banach once again, pick $\psi\in S_{X^ *}$ with   $\psi(x_0) = 0$ and $\psi(y) = \dist(y,\R x_0)$. {Then $\psi\vert_{[0,1]x_0}=0$ and by Proposition \ref{prop.APPP.2.6} we get that for every $m\in\F([0,1]x_0)$,
\begin{align*}\label{Eq.lem:2Dembedding.1}
    \|T(y)+{m}\| &\geq \langle \psi, T(y)+{m}  \rangle \\
    &=  \lim_{n,\U} \int_0^1 \partial_y (\phi_n * \psi)(sx_0) ds + \langle \psi, {m}  \rangle  \notag\\
    &= \int_0^1 \psi(y) ds + 0\notag\\
    &= \dist(y,\R x_0).\notag
\end{align*}}
\end{proof}

We will use the previous lemma together with the support-preservation property of $L$-projections given by Theorem~\ref{thm:complete_compact} in order to prove the following lemma.

\begin{lemma} \label{lem:l1basis}
Let $X$ be a Banach space such that  $\F(X)$ is $L$-embedded in $\F(X)^{**}$. If  $x,y\in S_X$ are such that $\dist(y,\R x) = 1$, then   $\{x,y\}$ is isometrically equivalent to the unit vector basis of $\ell_1^2$, i.e., 
\[ \|ax+by\|=|a|+|b|\ \text{ for all }\ a,b\in \R.\]
\end{lemma}

\begin{proof}
Fix an   $L$-projection $P: \F(X)^{**}\to\F(X)$,  let $x,y\in S_X$ be as in the statement, and let $Y  \coloneqq  \spn\{x,y\}$. Pick an operator $T: Y \to \F(Y)^{**}$  satisfying the conclusion of
Lemma~\ref{lem:2Dembedding} with $x_0 = x$.  Then, by Lemma~\ref{lem:2Dembedding}\ref{lem:2Dembedding.item2}, $T(y) \in \mathfrak{S}_Y([0,1]x) $. Therefore, noticing that  $\mathfrak{S}_Y([0,1]x)$ is contained in  $ \mathfrak{S}_X([0,1]x) $,  Theorem~\ref{thm:complete_compact}  gives that $P(T(y)) \in \F([0,1]x)$. Then, by Lemma~\ref{lem:2Dembedding}\ref{lem:2Dembedding.item4}, we have
\begin{align*}
    \|P(T(y))\| + 1 &= \|P(T(y))\| + \dist(y,\R x) \\
    &\leq \|P(T(y))\| + \|T(y) - P(T(y))\| \\
    &= \|T(y)\| \\
    &= \|y\| \\
    &= 1,
\end{align*}
which   shows that $P(T(y)) = 0$. Hence, {by Lemma~\ref{lem:2Dembedding}\ref{lem:2Dembedding.item1},\ref{lem:2Dembedding.item3},}  for any $a,b \in \R$, it follows that
\begin{align*}
    \|ax+by\| &= \|T(ax+by)\| \\
    &= \|aP(T(x)) + bP(T(y))\| + \|aT(x)+bT(y) - aP(T(x)) - bP(T(y))\| \\
    &= \|aT(x)\| + \|bT(y)\| \\
    &= |a|+|b|.
\end{align*}
This completes the proof.
\end{proof}

 We now prove our main result.

 \begin{theorem}
Let $X$ be a Banach space with $\dim (X) \geq 2$. Then $\F(X)$ is not $L$-embedded in $\F(X)^{**}$.\label{Thm.NonLEmb.Banach.dim.2} 
\end{theorem}

\begin{proof}
Assume towards a contradiction that $\F(X)$ is $L$-embedded in $\F(X)^{**}$. As $\dim(X)\geq 2$, we can find  $x,y \in S_X$ with $\dist(y,\R x) = 1$ -- such $y$ can be found for any $x$ by restricting to a two-dimensional subspace of $X$ and applying Riesz lemma and a compactness argument. By Lemma~\ref{lem:l1basis}, $\{x,y\}$ is isometrically equivalent to the unit vector basis of $\ell_1^2$. Using this, a simple calculation shows that, for
\[x'  \coloneqq  \frac{x+y}{2} \ \text{ and }\  y'  \coloneqq  \frac{x-y}{2},\]
$\{x', y'\}$ is    isometrically equivalent to the unit vector basis of $\ell_\infty^2$. In particular, we must have that $x',y' \in S_X$ and $\dist(y',\R x') = 1$. However,  using Lemma~\ref{lem:l1basis} again, this implies that $\{x',y'\}$ is also isometrically equivalent to the unit vector basis of $\ell_1^2$; a contradiction, since the unit vector basis of $\ell_1^2$ is not isometrically equivalent to the unit vector basis of $\ell_\infty^2$.
\end{proof}

\section{\texorpdfstring{Low-distortion ball embeddings preclude $L$-embeddability}{Low-distortion ball embeddings preclude -embeddability}}
We refine the core arguments from \S\ref{SectionnonLembBanach2} to generalize Theorem \ref{Thm.NonLEmb.Banach.dim.2} for metric spaces containing a biLipshitz copy of the unit ball of a 2-dimensional normed space with distortion less than a universal constant $C_0>1$. As an application, we conclude that all reversible Finsler manifolds, in particular all Riemannian manifolds, of dimension at least two, with their induced metrics, have non-$L$-embedded Lipschitz free spaces.
 
  \begin{remark}
In Lemma  \ref{lem:2Dembedding}, the operator $T$ has image in $\mathcal{F}(B_X)^{**}$. This can be seen as follows. For each $n\in\mathbb{N}$, define $\tilde T_n:X\to\mathcal{F}(B_X)^{**}$ given by
 \[\langle f, \tilde T_n(y)\rangle = \int_{0}^{1-1/2n}\partial_y(\phi_n*f)(sx_0)ds.\]
 Because $\phi_n$ is supported in $(1/n)B_X$, on a sufficiently small neighbourhood of the set $[0,(1-1/2n)x_0]$, the function $\phi_n*f$ only depends on $f|_{B_X}$, therefore, $T_ny\in \mathcal{F}(B_X)^{**}$. It is easy to see that $\lim_{n,\mathcal{U}}\langle f, T_ny\rangle  = \lim_{n,\mathcal{U}}\langle f|_{B_X},\rangle \tilde T_ny\rangle$, for all $y\in X$ and $f\in \mathrm{Lip}_0(X)$. This will be used in the following proof.
  \end{remark}
  
\begin{lemma} \label{distorted.lem:l1basis}
Let $M$ be a metric space such that $\F(M)$ is $L$-embedded in $\F(M)^{**}$. Let $X$ be a normed space such that there exists a $C$-biLipschitz copy of $B_X$ in $M$, for some $C<2$. Let $x,y\in S_X$ such that $\dist(y,\R x) > C-1$, and denote $\theta = \dist(y,\R x)$. Then $\{x,y\}$ is $C/(1+\theta-C)$-equivalent to the unit vector basis of $\ell_1^2$. More precisely, for all $a,b\in\R$
\begin{equation}
\label{distorted.lem:l1basis-equation}
\frac{1+\theta-C}{C} \big(|a|+|b|\big)\leq \|ax+by\| \leq |a|+|b|.
\end{equation}
\end{lemma}
  
\begin{proof}
 Fix a map $j:B_X\to M$ such that, for some $A,B>0$ with $AB\leq C$, we have that for all $x,y\in B_X$,
\[A^{-1}\|x-y\|\leq d(j(x),j(y)) \leq B\|x-y\|.\]
Translating the base point in $M$ does not change the isometric structure of $\mathcal{F}(M)$. We may therefore assume, without loss of generality, that $j(0_X) = 0_M$. Fix an $L$-projection $P: \F(M)^{**}\to\F(M)$ and,  let $x,y\in S_X$ be as in the statement, and let $Y  \coloneqq  \spn\{x,y\}$. Consider the linear extension of the Lipschitz map $j|_{B_Y}:B_Y\to j(B_Y)\sbs M$, $J:\mathcal{F}(B_Y) \to \mathcal{F}(j(B_Y))$, which is an onto isomorphism and satisfies $\|J^{-1}\|\leq A$, $\|J\|\leq B$. Pick an operator $T: Y \to \F(B_Y)^{**}$  satisfying the conclusion of Lemma~\ref{lem:2Dembedding} with $x_0 = x$. Then, by Lemma~\ref{lem:2Dembedding}\ref{lem:2Dembedding.item2}, $J^{**}T(y) \in \mathfrak{S}_{j(B_Y)}\big(j([0,1]x)\big) $. Noticing that  $\mathfrak{S}_{j(B_Y)}\big(j([0,1]x)\big) $ is contained in  $ \mathfrak{S}_M\big(j([0,1]x)\big)  $,  Theorem~\ref{thm:complete_compact}  gives that $PJ^{**}T(y) \in \F\big(j([0,1]x)\big)$. Consequently, $J^{-1}PJ^{**}T(y) \in \F([0,1]x)$. Then, by Lemma~\ref{lem:2Dembedding}\ref{lem:2Dembedding.item4}, we have
\begin{align*}
    \|PJ^{**}T(y)\| +  A^{-1}\theta &\leq \|PJ^{**}T(y)\| + A^{-1}\dist(y,\R x) \\
    &\leq \|PJ^{**}T(y)\| + A^{-1}\|T(y) - J^{-1}PJ^{**}T(y)\| \\
    &\leq \|PJ^{**}T(y)\| + \|J^{**}T(y) - PJ^{**}T(y)\| \\
    &= \|J^{**}T(y)\| \leq B\|y\| = B,
\end{align*}
which   shows that $\|PJ^{**}T(y)\| \leq B-A^{-1}\theta$. By Lemma~\ref{lem:2Dembedding}\ref{lem:2Dembedding.item3} we have $JTx\in \mathcal{F}(M)$. Hence, {by Lemma~\ref{lem:2Dembedding}\ref{lem:2Dembedding.item1},}  for any $a,b \in \R$, it follows that
\begin{align*}
    \|ax+by\| &= \|T(ax+by)\| \geq B^{-1}\|J^{**}T(ax+by)\| = B^{-1}\|aJTx + bJ^{**}Ty\|\\
    &= B^{-1}\|aPJT(x) + bPJ^{**}T(y)\|\\
    & + B^{-1}\big\|\big(aJT(x) - aPJT(x)\big) +\big(bJ^{**}T(y) - bPJ^{**}T(y)\big)\big\| \\
    & = B^{-1}\|aJT(x) + bPJ^{**}T(y)\| + B^{-1}\|bJ^{**}T(y) - bPJ^{**}T(y)\|\\
    &\geq B^{-1}\Big(|a|\|JT(x)\| - |b|(B-A^{-1}\theta) \Big) + B^{-1}\Big(|b|\|J^{**}T(y)\| - |b|(B-A^{-1}\theta)\Big) \\
    &\geq (AB)^{-1}\big(|a|\|Tx\|+|b|\|Ty\|\big) - |b|\big((1-(AB)^{-1}\theta\big)\\
    &\geq (AB)^{-1}(|a|+|b|) - \big((1-(AB)^{-1}\theta\big)(|a|+|b|) = ((1+\theta)C^{-1}-1\big)(|a|+|b|).
\end{align*}
This completes the proof.
\end{proof}

 \begin{theorem}\label{low-distortion.Thm.NonLEmb.Banach.dim.2}
There exists a universal constant $C_0 >1$ with the following property. Let $M$ be a metric space and $X$ be a normed space of dimension at least two. If $M$ contains a $C$-biLipschitz copy of $B_X$ with $C<C_0$, then $\F(M)$ is not $L$-embedded in $\F(M)^{**}$.
\end{theorem}

\begin{proof}
Choose the minimal $C_0 \in (1,\sqrt{2})$ such that $\Big(\frac{2}{C_0^2}-1\Big)2 \leq \frac{C_0}{2-C_0}$. Such a value is easily seen to exist by the Intermediate Value Theorem. Let $X$ be a normed space of dimension at least $2$, let $C < C_0$, and let $M$ be a metric space containing a $C$-biLipschitz copy of $B_X$. Assume towards a contradiction that $\F(M)$ is $L$-embedded in $\F(M)^{**}$. Because $\dim(X)\geq 2$, we can find  $x,y \in S_X$ with $\dist(y,\R x) = 1$. By Lemma~\ref{distorted.lem:l1basis}, \eqref{distorted.lem:l1basis-equation} holds for $\theta = 1$. Consider \[x'  \coloneqq  \frac{x+y}{2} \ \text{ and }\  y'  \coloneqq  \frac{x-y}{2}.\]
For all $\mu,\lambda\in\mathbb{R}$, we have by the triangle inequality that
\begin{equation}
\label{low-distortion.Thm.NonLEmb.Banach.dim.2-equation}
\|\mu y'+\lambda x'\| \leq \frac{1}{2}\Big(|\mu+\lambda|+|\mu-\lambda|\Big) = \max\{|\mu|,|\lambda|\},
\end{equation}
and we have by \eqref{distorted.lem:l1basis-equation} that
\begin{equation} \label{lower-bound}
\|\mu y'+\lambda x'\| \geq \frac{2-C}{C}\frac{1}{2}\Big(|\mu+\lambda|+|\mu-\lambda|\Big) = \frac{2-C}{C}\max\{|\mu|,|\lambda|\}.
\end{equation}

Define $y'' = y'/\|y'\|$ and $x'' = x'/\|x'\|$, which are in $S_X$ and satisfy
\[\mathrm{dist}(y'',\mathbb{R} x'') =  \|y'\|^{-1}\mathrm{dist}(y',\mathbb{R} x')\overset{\eqref{low-distortion.Thm.NonLEmb.Banach.dim.2-equation},\eqref{lower-bound}}{\geq}(2-C)/C = :\theta.\]
The condition $C<\sqrt 2$ already implies $\theta > C-1$. By Lemma~\ref{distorted.lem:l1basis}, the pair $\{x'',y''\}$ satisfies \eqref{distorted.lem:l1basis-equation} with constant $(1+\theta-C)/C = 2/C^2-1$. In particular, for all $\lambda,\mu\in\mathbb{R}$, we have
\[\|\mu x''+\lambda y''\| \geq \Big(\frac{2}{C^2}-1\Big)(|\mu|+|\lambda|).\]
On the other hand, by \eqref{low-distortion.Thm.NonLEmb.Banach.dim.2-equation},
\[\|\mu x''+\lambda y''\| \leq \max\{|\mu|\|x'\|^{-1},|\lambda|\|y'\|^{-1}\}\leq \frac{C}{2-C}\max\{|\mu|,|\lambda|\}.\]
Taking $\mu = \lambda = 1$, we obtain
\[\Big(\frac{2}{C^2}-1\Big)2 \leq \frac{C}{2-C}.\]
This contradicts the definition of $C_0$, since $C < C_0$.

\end{proof}

By \cite[Theorem 3.5]{luvcic:pasqualetto} (see also \cite{deng:hou}), if $M$ is a reversible Finsler manifold, $x\in M$, and $X = T_xM$ is the normed space associated with $x$, then, for every $\varepsilon>0$, there is a subset of $M$ containing $x$ that is $(1+\varepsilon)$-biLipschitzly isomorphic to $B_X$. By Theorem \ref{low-distortion.Thm.NonLEmb.Banach.dim.2}, we conclude the following.

\begin{corollary}
\label{finsler.lipschitz.free.space.non.L-embeddability}
Let $M$ be a reversible Finsler manifold of dimension at least two. Then $\mathcal{F}(M)$ is not $L$-embedded in $\mathcal{F}(M)^{**}$.
\end{corollary}

The above holds in particular for Riemannian manifolds of dimension at least two, since they are reversible Finsler manifolds.

\begin{acknowledgement*} This paper was written under the auspices of the
American Institute of Mathematics (AIM) SQuaREs program as part of the
``Nonlinear geometry of Banach spaces''  SQuaRE project.
\end{acknowledgement*}

\bibliographystyle{alpha}
\bibliography{bibliography}

\end{document}